\documentclass[a4paper,12pt]{article}
\title{{\bf Moduli spaces of stable quotients and 
wall-crossing phenomena}}
\date{}
\author{Yukinobu Toda}

\usepackage{makeidx}

\usepackage{latexsym}
\usepackage{amscd}
\usepackage{amsmath}
\usepackage{amssymb}
\usepackage{amsthm}
\usepackage{float}
\usepackage{graphicx}

\usepackage[all,ps,dvips]{xy}

\usepackage{array}
\usepackage{amscd}
\usepackage[all]{xy}
\usepackage{makeidx}
\usepackage{latexsym}
\DeclareFontFamily{U}{rsfs}{%
\skewchar\font127}
\DeclareFontShape{U}{rsfs}{m}{n}{%
<-6>rsfs5<6-8.5>rsfs7<8.5->rsfs10}{}
\DeclareSymbolFont{rsfs}{U}{rsfs}{m}{n}
\DeclareSymbolFontAlphabet
{\mathrsfs}{rsfs}
\DeclareRobustCommand*\rsfs{%
\@fontswitch\relax\mathrsfs}
\theoremstyle{plain}
\newtheorem{thm}{Theorem}[section]
\newtheorem{prop}[thm]{Proposition}
\newtheorem{lem}[thm]{Lemma}

\newtheorem{defi}[thm]{Definition}
\newtheorem{rmk}[thm]{Remark}

\newtheorem{prob}[thm]{Problem}

\newtheorem{prop-defi}[thm]{Proposition-Definition}
\newtheorem{thm-defi}[thm]{Theorem-Definition}
\newtheorem{lem-defi}[thm]{Lemma-Definition}

\newtheorem{exam}[thm]{Example}

\newdimen\argwidth
\def\db[#1\db]{
 \setbox0=\hbox{$#1$}\argwidth=\wd0
 \setbox0=\hbox{$\left[\box0\right]$}
  \advance\argwidth by -\wd0
 \left[\kern.3\argwidth\box0 \kern.3\argwidth\right]}

\newcommand{\cC}{\mathcal{C}}

\newcommand{\eE}{\mathcal{E}}

\newcommand{\hH}{\mathcal{H}}

\newcommand{\lL}{\mathcal{L}}
\newcommand{\mM}{\mathcal{M}}

\newcommand{\oO}{\mathcal{O}}

\newcommand{\qQ}{\mathcal{Q}}

\newcommand{\sS}{\mathcal{S}}

\newcommand{\xX}{\mathcal{X}}

\newcommand{\Supp}{\mathop{\rm Supp}\nolimits}
\newcommand{\Hom}{\mathop{\rm Hom}\nolimits}

\newcommand{\dR}{\mathbf{R}}

\newcommand{\Hilb}{\mathop{\rm Hilb}\nolimits}

\newcommand{\Pic}{\mathop{\rm Pic}\nolimits}

\newcommand{\Isom}{\mathop{\rm Isom}\nolimits}

\newcommand{\Tan}{\mathop{\rm Tan}\nolimits}
\newcommand{\Obs}{\mathop{\rm Obs}\nolimits}

\newcommand{\id}{\textrm{id}}

\newcommand{\rk}{\mathop{\rm rk}\nolimits}
\newcommand{\Cont}{\mathop{\rm Cont}\nolimits}

\newcommand{\Ext}{\mathop{\rm Ext}\nolimits}
\newcommand{\Spec}{\mathop{\rm Spec}\nolimits}
\newcommand{\rank}{\mathop{\rm rank}\nolimits}

\newcommand{\ev}{\mathop{\rm ev}\nolimits}

\newcommand{\valence}{\mathop{\rm val}\nolimits}

\newcommand{\vir}{\mathop{\rm vir}\nolimits}

\newcommand{\cneq}{\mathrel{\raise.095ex\hbox{:}\mkern-4.2mu=}}
\newcommand{\eqcn}{\mathrel{=\mkern-4.5mu\raise.095ex\hbox{:}}}

\newcommand{\Cok}{\mathop{\rm Cok}\nolimits}

\newcommand{\Aut}{\mathop{\rm Aut}\nolimits}

\newcommand{\SL}{\mathop{\rm SL}\nolimits}

\newcommand{\Sch}{\mathop{\rm Sch}\nolimits}

\newcommand{\pt}{\mathop{\rm pt}\nolimits}
\newcommand{\Sym}{\mathop{\rm Sym}\nolimits}

\newcommand{\Quot}{\mathop{\rm Quot}\nolimits}

\newcommand{\Ker}{\mathop{\rm ker}\nolimits}

\newcommand{\GL}{\mathop{\rm GL}\nolimits}
\newcommand{\PGL}{\mathop{\rm PGL}\nolimits}

\newcommand{\length}{\mathop{\rm length}\nolimits}

\begin{document}
\maketitle

\begin{abstract}
The moduli space of holomorphic maps from Riemann surfaces
to the Grassmannian is known to have two kinds of 
compactifications: Kontsevich's stable map compactification
and Marian-Oprea-Pandharipande's stable quotient compactification. 
Over a non-singular curve, the latter moduli space is 
Grothendieck's Quot scheme. 
In this paper, we give the notion of `$\epsilon$-stable 
quotients' for a positive real number $\epsilon$, and  
show that
stable maps and stable quotients are related by wall-crossing 
phenomena.
We will also discuss Gromov-Witten type 
invariants associated to 
$\epsilon$-stable quotients, and investigate
them under wall-crossing. 
\end{abstract}
\section{Introduction}
The purpose of this paper is 
to investigate wall-crossing phenomena 
of several compactifications of the 
moduli spaces of holomorphic maps from Riemann surfaces to the 
Grassmannian. So far, two kinds of compactifications 
are known: Kontsevich's
stable map compactification~\cite{Ktor} and
Marian-Oprea-Pandharipande's stable quotient
compactification~\cite{MOP}. 
The latter moduli space is introduced rather recently, 
and it is Grothendieck's Quot scheme over a 
non-singular curve. 
In this paper, we will introduce the notion of 
\textit{$\epsilon$-stable quotients}
for a positive real number $\epsilon \in \mathbb{R}_{>0}$, 
and show that 
the moduli space of $\epsilon$-stable quotients
is a proper Deligne-Mumford stack
over $\mathbb{C}$
with a perfect obstruction theory. 
It will turn out that there is a wall and 
chamber structure on the space of stability 
conditions $\epsilon \in \mathbb{R}_{>0}$, and the moduli 
spaces are constant at chambers but jump at walls, i.e. 
wall-crossing phenomena occurs. 
 We will see that 
stable maps and stable quotients 
are related by the above wall-crossing phenomena. 
We will also consider the virtual fundamental classes on 
the moduli spaces of $\epsilon$-stable quotients, 
the associated enumerative invariants,
and investigate them 
under the change of $\epsilon \in \mathbb{R}_{>0}$.
This is interpreted as a wall-crossing formula of 
Gromov-Witten (GW) type invariants.

\subsection{Stable maps and stable quotients}
Let $C$ be a smooth projective curve
 over 
$\mathbb{C}$ of genus $g$, 
and $\mathbb{G}(r, n)$ the 
Grassmannian which parameterizes $r$-dimensional $\mathbb{C}$-vector 
subspaces in $\mathbb{C}^n$.
Let us consider a holomorphic map
\begin{align}\label{map}
f\colon C \to \mathbb{G}(r, n), 
\end{align}
satisfying the following,
\begin{align*}
f_{\ast}[C]=d \in H_2(\mathbb{G}(r, n), \mathbb{Z})\cong \mathbb{Z}.
\end{align*}
By the universal property of $\mathbb{G}(r, n)$, 
giving a map (\ref{map}) is equivalent to 
giving a quotient,
\begin{align}\label{exseq}
\oO_C^{\oplus n} \twoheadrightarrow Q, 
\end{align}
where $Q$ is a locally free sheaf of rank $n-r$
and degree $d$. 
The moduli space of maps (\ref{map})
is not compact, and 
two kinds of compactifications are known:
compactification as maps (\ref{map}) or compactification as quotients
(\ref{exseq}). 
\begin{itemize}
\item {\bf Stable map compactification:}
We attach trees of rational curves to $C$, and 
consider moduli space of 
maps from the attached nodal curves to $\mathbb{G}(r, n)$
with finite automorphisms. 
\item {\bf Quot scheme compactification:}
We consider the moduli space of quotients (\ref{exseq}), 
allowing torsion subsheaves in $Q$. 
The resulting moduli space is Grothendieck's Quot scheme on $C$. 
\end{itemize}
In the above compactifications,  
the (stabilization of the)
source curve $C$ is fixed in the moduli. 
If we vary the curve $C$ as a nodal curve
and give $m$-marked points on it, 
we obtain two kinds of compact moduli spaces, 
\begin{align}\label{Stmap}
&\overline{M}_{g, m}(\mathbb{G}(r, n), d), \\
\label{Stquo}
&\overline{Q}_{g, m}(\mathbb{G}(r, n), d). 
\end{align}
The space (\ref{Stmap}) is a moduli 
space 
of Kontsevich's \textit{stable maps}~\cite{Ktor}.
Namely this is the 
moduli space 
of data, 
\begin{align*}
(C, p_1, \cdots, p_m, f\colon C\to \mathbb{G}(r, n)), 
\end{align*}
where $C$ is a genus $g$, $m$-pointed 
nodal curve and 
$f$ is a morphism with finite automorphisms. 

The space (\ref{Stquo}) is a moduli space of 
Marian-Oprea-Pandharipande's stable quotients~\cite{MOP}, 
 which we call \textit{MOP-stable quotients}.
 By definition a
 MOP-stable quotient consists of data,
 \begin{align}\label{data}
 (C, p_1, \cdots, p_m, \oO_C^{\oplus n} \stackrel{q}{\twoheadrightarrow} Q),
 \end{align}
 for an $m$-pointed nodal curve $C$ and a quotient 
sheaf $Q$ on it, satisfying the following
stability condition. 
\begin{itemize}
\item The coherent sheaf $Q$ is locally free near 
nodes and markings. 
In particular, the determinant 
line bundle $\det (Q)$ is well-defined. 
\item The $\mathbb{R}$-line bundle 
\begin{align}\label{Rline}
\omega_C(p_1 +\cdots +p_m)\otimes \det(Q)^{\otimes \epsilon},
\end{align}
is ample for \textit{every} $\epsilon>0$. 
\end{itemize}
 The space (\ref{Stquo}) is the moduli space of 
 MOP-stable quotients (\ref{data}) 
with $C$ genus $g$, 
 $\rank(Q)=n-r$ and $\deg (Q)=d$. 
Both moduli spaces (\ref{Stmap}) and (\ref{Stquo})
have the following properties. 
\begin{itemize}
\item 
The moduli spaces (\ref{Stmap})
and (\ref{Stquo}) are proper
 Deligne-Mumford stacks over $\mathbb{C}$ 
with perfect obstruction theories~\cite{BGW}, \cite{MOP}. 
 \item The moduli spaces (\ref{Stmap}), (\ref{Stquo}) 
carry proper morphisms, 
\begin{align}\label{MQdia}
\xymatrix{
\overline{M}_{g, m}(\mathbb{G}(r, n), d) \ar[dr] & & 
\overline{Q}_{g, m}(\mathbb{G}(r, n), d) \ar[dl] \\
& \overline{M}_{g, m}.&
}
\end{align}
\end{itemize}
Here $\overline{M}_{g, m}$ is the moduli space 
of genus $g$, $m$-pointed stable curves. 
Taking the fibers of the 
diagram (\ref{MQdia}) 
over a non-singular curve $[C] \in \overline{M}_{g, 0}$, 
we obtain the compactifications as maps (\ref{map}),
quotients (\ref{exseq}) respectively. 
Also the associated virtual fundamental classes
 on the moduli spaces (\ref{Stmap}), (\ref{Stquo}) are 
compared in~\cite[Section~7]{MOP}. 
\subsection{$\epsilon$-stable quotients}
The purpose of this paper is to introduce 
 a variant of 
stable quotient theory, depending 
on a positive real number, 
\begin{align}\label{fixed}
\epsilon \in \mathbb{R}_{>0}. 
\end{align}
We define an \textit{$\epsilon$-stable quotient} 
to be data (\ref{data}), 
which has the same property to 
MOP-stable quotients 
except the following. 
\begin{itemize}
\item The $\mathbb{R}$-line bundle (\ref{Rline}) 
is only ample with respect to the \textit{fixed} 
stability parameter $\epsilon\in \mathbb{R}_{>0}$. 
\item For any $p\in C$, the torsion subsheaf $\tau(Q)\subset Q$ satisfies 
\begin{align*}
\epsilon \cdot \length \tau(Q)_{p} \le 1.
\end{align*} 
\end{itemize}
The idea of $\epsilon$-stable quotients
originates from Hassett's weighted pointed stable 
curves. 
In~\cite{BH}, Hassett introduces the notion of 
weighted pointed stable curves 
$(C, p_1, \cdots, p_m)$, where 
$C$ is a nodal curve and $p_i \in C$
are marked points. 
The stability condition depends on a choice 
of a weight, 
\begin{align}\label{intro:weight}
(a_1, a_2, \cdots, a_m) \in (0, 1]^m,
\end{align}
which put a similar constraint 
for the pointed curve $(C, p_1, \cdots, p_m)$
to our $\epsilon$-stability. 
(See Definition~\ref{def:weighted}.)
A choice of $\epsilon$ 
in our situation corresponds to a choice 
of a weight (\ref{intro:weight})
 for weighted pointed stable curves. 

 The moduli space of 
 $\epsilon$-stable quotients (\ref{data}) 
with $C$ genus $g$, 
 $\rank(Q)=n-r$ and $\deg (Q)=d$ is denoted 
 by 
 \begin{align}\label{eStquo}
 \overline{Q}_{g, m}^{\epsilon}(\mathbb{G}(r, n), d). 
 \end{align}
We show the following result. 
(cf.~Theorem~\ref{thm:rep}, 
Subsection~\ref{subsec:Pro},
Proposition~\ref{prop:wall}, 
Proposition~\ref{wall2}, Theorem~\ref{thm:cross}.)
\begin{thm}\label{thm:main}
(i) The moduli space 
$\overline{Q}_{g, m}^{\epsilon}(\mathbb{G}(r, n), d)$
is a proper Deligne-Mumford stack 
 over $\mathbb{C}$ with a 
perfect obstruction theory. 
Also there is a proper morphism, 
\begin{align}\label{QMp}
\overline{Q}_{g, m}^{\epsilon}(\mathbb{G}(r, n), d)
\to \overline{M}_{g, m}. 
\end{align} 

(ii) There is a finite number of values 
\begin{align*}
0=\epsilon_0<\epsilon_1< \cdots < \epsilon_k<\epsilon_{k+1}=\infty,
\end{align*}
such that we have 
\begin{align*}
\overline{Q}_{g, m}^{\epsilon}(\mathbb{G}(r, n), d)=
\overline{Q}_{g, m}^{\epsilon_i}(\mathbb{G}(r, n), d),
\end{align*}
for  $\epsilon \in (\epsilon_{i-1}, \epsilon_{i}]$. 

(iii) We have the following. 
\begin{align*}
\overline{Q}_{g, m}^{\epsilon}(\mathbb{G}(r, n), d)
&\cong 
\overline{M}_{g, m}(\mathbb{G}(r, n), d), \quad  
\epsilon >2, \\
\overline{Q}_{g, m}^{\epsilon}(\mathbb{G}(r, n), d)
&\cong 
\overline{Q}_{g, m}(\mathbb{G}(r, n), d), \quad  
0<\epsilon \le 1/d.  
\end{align*}
\end{thm}
By Theorem~\ref{thm:main} (i), 
there is the associated virtual fundamental class, 
\begin{align*}
[\overline{Q}_{g, m}^{\epsilon}(\mathbb{G}(r, n), d)]^{\rm{vir}}
\in A_{\ast}
(\overline{Q}_{g, m}^{\epsilon}(\mathbb{G}(r, n), d), \mathbb{Q}).
\end{align*}
A comparison of the above virtual fundamental classes 
under change of $\epsilon$ is obtained as follows. 
(cf.~Theorem~\ref{thm:wcf}.)
\begin{thm}
For $\epsilon \ge \epsilon' >0$ satisfying
$2g-2+\epsilon' \cdot d>0$, there is a diagram, 
\begin{align*}
\xymatrix{
  \overline{Q}_{g, m}^{\epsilon}(\mathbb{G}(r, n), d) 
\ar[r]^{\iota^{\epsilon}} & 
\overline{Q}_{g, m}^{\epsilon}(\mathbb{G}(1,\dbinom{n}{r} ), d)
\ar[d] _{c_{\epsilon, \epsilon'}}, \\
  \overline{Q}_{g, m}^{\epsilon'}(\mathbb{G}(r, n), d) 
\ar[r]^{\iota^{\epsilon'}} & 
\overline{Q}_{g, m}^{\epsilon'}(\mathbb{G}(1, \dbinom{n}{r}), d),
}
\end{align*}
such that we have 
\begin{align*}
c_{\epsilon, \epsilon' \ast}\iota_{\ast}^{\epsilon}
[\overline{Q}_{g, m}^{\epsilon}(\mathbb{G}(r, n), d)]^{\rm vir}
=\iota_{\ast}^{\epsilon'}
[\overline{Q}_{g, m}^{\epsilon'}(\mathbb{G}(r, n), d)]^{\rm vir}
\end{align*}
\end{thm}

The above theorem, which is a refinement 
of the result in~\cite[Section~7]{MOP},
 is interpreted as a wall-crossing 
formula relevant to the GW theory. 
\subsection{Invariants on Calabi-Yau 3-folds}
The idea of $\epsilon$-stable quotients 
is also applied to define new 
quantum invariants on some compact or non-compact Calabi-Yau 
3-folds. 
One of the interesting examples is 
a system of invariants on 
a quintic Calabi-Yau 3-fold 
$X\subset \mathbb{P}^4$. 
In Section~\ref{sec:enu}, we associate the substack, 
\begin{align*}
\overline{Q}_{0, m}^{\epsilon}(X, d) \subset 
\overline{Q}_{0, m}^{\epsilon}(\mathbb{P}^4, d),
\end{align*}
such that when $\epsilon>2$, it coincides with the moduli space of 
genus zero, degree $d$
stable maps to $X$. 
There is a
perfect obstruction theory on the space
$\overline{Q}_{0, m}^{\epsilon}(X, d)$, 
hence the virtual class, 
\begin{align*}
[\overline{Q}_{0, m}^{\epsilon}(X, d)]^{\rm vir}
\in A_{\ast}(\overline{Q}_{0, m}^{\epsilon}(X, d), \mathbb{Q}), 
\end{align*}
with virtual dimension $m$. 
In particular, the zero-pointed moduli space yields 
the invariant, 
\begin{align*}
N_{0, d}^{\epsilon}(X)=
\int_{[\overline{Q}_{0, 0}^{\epsilon}(X, d)]^{\rm vir}}1 \in \mathbb{Q}. 
\end{align*} 
For $\epsilon>2$, the invariant $N_{0, d}^{\epsilon}(X)$
coincides with the 
GW invariant counting genus zero, degree $d$ stable 
maps to $X$.  
However for a smaller $\epsilon$, the above 
invariant may be different from the GW invariant
of $X$. 
The understanding of 
wall-crossing phenomena of such invariants seem relevant
to the study of the GW theory. 
In Section~\ref{sec:enu}, 
we will also discuss 
such invariants in several other cases. 

\subsection{Relation to other works}
As pointed out in~\cite[Section~1]{MOP},
only a few proper moduli spaces carrying virtual 
classes are known, e.g. stable maps~\cite{BGW}, stable sheaves on 
surfaces or 3-folds~\cite{LTV},~\cite{Thom}, 
Grothendieck's Quot scheme on non-singular curves~\cite{MO}
and MOP-stable quotients~\cite{MOP}. By the result of
Theorem~\ref{thm:main}, we have constructed 
a new family of moduli spaces which have 
virtual classes. 

Before the appearance of stable maps~\cite{Ktor}, the Quot scheme 
was used for an enumeration problem of curves on 
the Grassmannian~\cite{Bert1}, \cite{Bert2}, \cite{BDW}.
Some relationship between 
compactifications as maps (\ref{map})
and quotients (\ref{exseq})
is discussed in~\cite{PR}.
The fiber of the morphism (\ref{QMp}) 
over a non-singular curve 
is an intermediate moduli space
between the above two
compactifications. 
This fact seems to give a new insight to 
the work~\cite{PR}. 

Wall-crossing phenomena for stable maps 
or GW type invariants are 
discussed in~\cite{BH}, \cite{BaMa}, \cite{GuAl}. 
In their works, a stability 
condition
is a weight on the 
marked points, not on maps. 
In particular, there is no wall-crossing 
phenomena if there is no point insertion. 

After the author finished the work of this paper, 
a closely related work
of Mustat$\check{\rm{a}}$-Mustat$\check{\rm{a}}$~\cite{MMu}
was informed to the author.  
 They  construct some compactifications 
 of the moduli space of maps from 
 Riemann surfaces to the projective space, 
 which are interpreted as moduli 
spaces of $\epsilon$-stable quotients of rank one. 
However they do not 
address higher rank quotients, virtual classes
nor wall-crossing formula. 
In this sense, the prewent work
is interpreted as 
a combination of the works~\cite{MOP} and~\cite{MMu}. 

Recently wall-crossing formula of Donaldson-Thomas (DT) 
type invariants have been developed by Kontsevich-Soibelman~\cite{K-S}
and Joyce-Song~\cite{JS}.
The DT invariant is a counting invariant
of stable sheaves on a Calabi-Yau 3-fold, 
while GW invariant is a counting invariant 
of stable maps. 
The relationship between GW invariants and DT invariants
is proposed by Maulik-Nekrasov-Okounkov-Pandharipande
(MNOP)~\cite{MNOP},
called \textit{GW/DT correspondence}.
In the DT side, a number of applications
of wall-crossing formula 
 to the MNOP conjecture 
have been found recently,  
such as \textit{DT/PT-correspondence}, \textit{rationality 
conjecture}. (cf.~\cite{BrH}, \cite{StTh}, \cite{Tcurve1}, \cite{Tolim2}.)
It seems worth trying to find a similar wall-crossing phenomena 
in GW side and give an application to the MNOP conjecture. 
The work of this paper grew out from such an attempt. 

\subsection{Acknowledgement}
The author thanks V. Alexeev for pointing out 
the related work~\cite{MMu}, 
and Y.~Konoshi for the information of the reference~\cite{AMV}. 
This work is supported by World Premier International 
Research Center Initiative (WPI initiative), MEXT, Japan. 
This work is also supported by Grant-in Aid
for Scientific Research grant numbers 22684002 from the 
Ministry of Education, Culture,
Sports, Science and Technology, Japan.

\section{Stable quotients}
In this section we introduce the notion 
of $\epsilon$-stable quotients
for a positive real number $\epsilon\in \mathbb{R}_{>0}$, 
study their properties, and give some examples. 
The $\epsilon$-stable quotients 
are extended notion of stable quotients 
introduced by Marian-Oprea-Pandharipande~\cite{MOP}. 
\subsection{Definition of $\epsilon$-stable quotients}
Let $C$ be a connected 
projective curve over $\mathbb{C}$ with 
at worst nodal singularities. 
Suppose that the arithmetic genus of $C$ is $g$, 
\begin{align*}
g=\dim H^1(C, \oO_C).
\end{align*}
Let $C^{ns}\subset C$ be the non-singular locus of $C$. 
We say the data 
\begin{align*}
(C, p_1, \cdots, p_m), 
\end{align*}
with distinct markings $p_i \in C^{ns}\subset C$
a genus $g$, $m$-pointed, \textit{quasi-stable curve}. 
The notion of quasi-stable quotients is introduced in~\cite[Section~2]{MOP}. 
\begin{defi}\emph{
Let $C$ be a pointed quasi-stable curve and 
$q$ a quotient, 
\begin{align*}
\oO_C^{\oplus n} \stackrel{q}{\twoheadrightarrow} Q.
\end{align*}
We say that $q$ is a \textit{quasi-stable quotient}
if $Q$ is locally free near nodes and markings. 
In particular, the torsion subsheaf $\tau(Q)\subset Q$ satisfies,
\begin{align*}
\Supp \tau(Q) \subset C^{ns} \setminus \{p_1, \cdots, p_m\}.
\end{align*}}
\end{defi}
Let $\oO_C^{\oplus n} \stackrel{q}{\twoheadrightarrow} Q$ be a quasi-stable 
quotient. 
The quasi-stability implies that the sheaf 
$Q$ is perfect, i.e. there is a finite locally free resolution
$P^{\bullet}$ of $Q$. 
In particular, the determinant line bundle, 
\begin{align*}
\det(Q)=\bigotimes_{i}(\bigwedge^{\rk P^{i}}P^i )^{\otimes (-1)^i} \in \Pic(C), \end{align*}
makes sense. 
The degree of $Q$ is defined by the degree of 
$\det(Q)$. 
 We say that a quasi-stable quotient $\oO_C^{\oplus n} \twoheadrightarrow Q$
is of \textit{type $(r, n, d)$}, if the following holds, 
\begin{align*}
\rank Q=n-r, \quad \deg Q=d. 
\end{align*}
For a quasi-stable quotient 
$\oO_C^{\oplus n}\stackrel{q}{\twoheadrightarrow}Q$
and $\epsilon \in \mathbb{R}_{>0}$, the 
$\mathbb{R}$-line bundle $\lL(q, \epsilon)$
is defined by 
\begin{align}\label{def:L}
\lL(q, \epsilon)\cneq \omega_C(p_1 +\cdots +p_m)
\otimes (\det Q)^{\otimes \epsilon}.
\end{align}
The notion of stable quotients introduced in~\cite{MOP}, 
which we call \textit{MOP-stable quotients}, 
is defined as follows. 
\begin{defi}\emph{{\bf \cite{MOP}}
A quasi-stable quotient 
$\oO_C^{\oplus n}\stackrel{q}{\twoheadrightarrow}Q$
is a \textit{MOP-stable quotient} if the $\mathbb{R}$-line 
bundle $\lL(q, \epsilon)$ is ample for every $\epsilon>0$.} 
\end{defi}
The idea of $\epsilon$-stable quotient is that, 
we only require the ampleness of $\lL(q, \epsilon)$
for a fixed $\epsilon$, (not every $\epsilon>0$,)
and put an additional condition on the length of the torsion 
subsheaf of the quotient sheaf. 
\begin{defi}\emph{
Let  
 $\oO_C^{\oplus n}\stackrel{q}{\twoheadrightarrow}
Q$ 
be a quasi-stable quotient and 
$\epsilon$ a positive real number. We say that $q$ is an
\textit{$\epsilon$-stable quotient} if the following conditions 
are satisfied. }
\begin{itemize}
\item \emph{The $\mathbb{R}$-line bundle
$\lL(q, \epsilon)$ 
is ample.}
\item \emph{For any point $p\in C$, 
the torsion subsheaf $\tau(Q)\subset Q$ satisfies 
the following inequality,} 
\begin{align}\label{ineq}
\epsilon \cdot \length \tau(Q)_{p} \le 1. 
\end{align}
\end{itemize}
\end{defi}
Here we give some remarks. 
\begin{rmk}
As we mentioned in the introduction, 
the definition of $\epsilon$-stable quotients
is motivated by Hassett's weighted  
pointed stable curves~\cite{BH}. 
We will discuss the relationship 
between $\epsilon$-stable quotients and 
weighted pointed stable curves in 
Subsection~\ref{subsec:Hassett}. 
\end{rmk}
\begin{rmk}
The ampleness of $\lL(q, \epsilon)$ for 
every $\epsilon>0$ is equivalent to the 
ampleness of $\lL(q, \epsilon)$ for $0<\epsilon \ll 1$. 
If $\epsilon>0$ is sufficiently small, 
then the condition (\ref{ineq}) does not say anything, 
so MOP-stable quotients 
coincide with $\epsilon$-stable quotients for $0<\epsilon \ll 1$. 
\end{rmk}
\begin{rmk}
For a quasi-stable quotient $\oO_C^{\oplus n}\twoheadrightarrow Q$, 
take the exact sequence, 
\begin{align*}
0 \to S \to \oO_C^{\oplus n} \to Q \to 0. 
\end{align*}
The quasi-stability implies that $S$ is locally free. 
By taking the dual of the above exact sequence,
giving a quasi-stable quotient is equivalent to giving 
a locally free sheaf $S^{\vee}$ and a morphism 
\begin{align*}
\oO_C^{\oplus n} \stackrel{s}{\to} S^{\vee}, 
\end{align*} 
which is surjective on nodes and marked points. 
The $\epsilon$-stability is also defined in terms of 
data $(S^{\vee}, s)$. 
\end{rmk}

\begin{rmk}\label{e1}
By definition, a
 quasi-stable quotient $\oO_{C}^{\oplus n}\stackrel{q}{\twoheadrightarrow} Q$
of type $(r, n, d)$ 
induces a rational map, 
\begin{align*}
f\colon C \dashrightarrow \mathbb{G}(r, n), 
\end{align*}
such that we have 
\begin{align}\label{eq:deg}
\deg f_{\ast}[C] +\length \tau(Q)=d. 
\end{align}
If $\epsilon >1$, then the condition (\ref{ineq}) 
is equivalent to that $Q$ is a locally free sheaf. 
Hence $f$ is an actual map, 
and the quotient $q$ is isomorphic to the pull-back 
of the universal quotient on $\mathbb{G}(r, n)$. 
\end{rmk}
Let $C$ be a marked quasi-stable curve. 
A point $p\in C$ is called \textit{special} if 
$p$ is a singular point of $C$ or 
a marked point. For an irreducible 
component $P\subset C$, we denote by $s(P)$ the 
number of special points in $P$. 
The following lemma is obvious. 
\begin{lem}\label{lem:ob}
Let $\oO_C^{\oplus n} \stackrel{q}{\twoheadrightarrow} Q$
be a quasi-stable quotient and take $\epsilon \in \mathbb{R}_{>0}$. 
Then the $\mathbb{R}$-line bundle $\lL(q, \epsilon)$ is 
ample if and only if 
for any irreducible component $P\subset C$
with genus $g(P)$, 
the following condition holds. 
\begin{align}\label{ob1}
&\deg (Q|_{P}) >0, \quad (s(P), g(P))=(2, 0), (0, 1) \\
\label{ob2}
&\deg (Q|_{P})>1/\epsilon, \quad (s(P), g(P))=(1, 0), \\
\label{ob3}
& \deg (Q|_{P})>2/\epsilon, \quad (s(P), g(P))=(0, 0). 
\end{align}
\end{lem}
\begin{proof}
For an irreducible component $P\subset C$, we have 
\begin{align*}
\deg (\lL(q, \epsilon)|_{P})
=2g(P)-2+s(P)+\epsilon \cdot \deg(Q|_{P}). 
\end{align*}
Also since $q$ is surjective, we have $\deg(Q|_{P})\ge 0$. 
Therefore the lemma follows. 
\end{proof}
Here we give some examples. 
We will discuss some more examples in 
Section~\ref{sec:ex}.
\begin{exam}
(i) Let $C$ be a smooth projective curve 
of genus $g$
and $f\colon C\to \mathbb{G}(r, n)$ a map. 
Suppose that $f$ is non-constant if $g\le 1$. 
By pulling back the universal quotient
\begin{align*}
\oO_{\mathbb{G}(r, n)}^{\oplus n} \twoheadrightarrow 
\qQ_{\mathbb{G}(r, n)}, 
\end{align*}
on $\mathbb{G}(r, n)$, 
we obtain the quotient $\oO_C^{\oplus n} \stackrel{q}{\twoheadrightarrow} Q$. 
It is easy to see that
the quotient $q$ is an $\epsilon$-stable quotient for $\epsilon >2$. 

(ii) Let $C$ be as in (i) and take 
distinct points $p_1, \cdots, p_m \in C$. 
For an effective divisor $D=a_1 p_1 +\cdots a_m p_m$
with $a_i>0$, the quotient
\begin{align*}
\oO_C \stackrel{q}{\twoheadrightarrow} \oO_D,
\end{align*} 
is an $\epsilon$-stable quotient if and only if 
\begin{align*}
2g-2+\epsilon \cdot \sum_{i=1}^{m}a_i>0, \quad 0<\epsilon \le 1/a_i, 
\end{align*}
for all $1\le i\le m$. 
In this case, the quotient $q$ is 
MOP-stable if 
 $g \ge 1$, but this is not the case in genus zero.  

(iii) Let $\mathbb{P}^1 \cong C\subset \mathbb{P}^n$ be a line 
and take distinct points $p_1, p_2 \in C$. 
By restricting the Euler sequence to $C$, we obtain the 
exact sequence, 
\begin{align*}
0 \to \oO_C(-1) \stackrel{s}{\to}
 \oO_C^{\oplus n+1} \to T_{\mathbb{P}^n}(-1)|_{C} \to 0.
\end{align*}
Composing the natural inclusion $\oO_C(-p_1-p_2 -1)\subset \oO_C(-1)$
with $s$, we obtain the exact sequence, 
\begin{align*}
0\to \oO_C(-p_1-p_2-1)\to \oO_C^{\oplus n+1} \stackrel{q}{\to} Q \to 0. 
\end{align*}
It is easy to see that the
 quotient $q$ is $\epsilon$-stable for $\epsilon=1$. 
Note that $q$ is not a MOP-stable quotient
nor a quotient corresponding to a stable map as in (i). 
\end{exam}

\subsection{Moduli spaces of $\epsilon$-stable quotients}
Here we define the moduli functor of the family 
of $\epsilon$-stable quotients. 
We use the language of stacks, and 
readers can refer~\cite{GL} for their introduction. 
First we recall the moduli stack of 
quasi-stable curves. 
For a $\mathbb{C}$-scheme $B$, a \textit{family of
genus $g$, 
$m$-pointed quasi-stable curves over $B$}
is defined to be data
\begin{align*}
(\pi \colon \cC \to B, p_1, \cdots, p_m), 
\end{align*}
which satisfies the following. 
\begin{itemize}
\item The morphism $\pi \colon \cC \to B$ is
flat, proper and locally of finite presentation. 
Its relative dimension is one and $p_1, \cdots, p_m$
are sections of $\pi$. 
\item For each closed point $b\in B$, the data 
\begin{align*}
(\cC_b \cneq \pi^{-1}(b), p_1(b), \cdots, p_m(b)),
\end{align*}
is an $m$-pointed quasi-stable curve. 
\end{itemize}
The families of genus $g$, $m$-pointed 
quasi-stable curves form a groupoid
$\mM_{g, m}(B)$
with the set of isomorphisms, 
\begin{align*}
\Isom_{\mM_{g, m}(B)}((\cC, p_1, \cdots, p_m), 
(\cC', p_1', \cdots, p_m')),
\end{align*}
given by the isomorphisms of schemes over $B$, 
\begin{align*}
\phi \colon \cC \stackrel{\cong}{\to}\cC', 
\end{align*}
satisfying $\phi(p_i)=p_i'$ for each $1\le i\le m$. 
The assignment $B\mapsto \mM_{g, m}(B)$ forms a 2-functor, 
\begin{align*}
\mM_{g, m}\colon \Sch/\mathbb{C} \to (\mathrm{groupoid}), 
\end{align*}
which is known to be an algebraic stack locally of finite type 
over $\mathbb{C}$. 
\begin{defi}\emph{
For a given data 
\begin{align*}
\epsilon \in \mathbb{R}_{>0}, \quad 
(r, n, d)\in \mathbb{Z}^{\oplus 3}, 
\end{align*}
we define the \textit{stack of 
genus $g$, $m$-pointed $\epsilon$-stable quotient of
type $(r, n, d)$} to be the 2-functor, 
\begin{align}\label{2func}
\overline{\qQ}^{\epsilon}_{g, m}(\mathbb{G}(r, n), d) \colon 
\Sch/\mathbb{C} \to (\mathrm{groupoid}), 
\end{align}
which sends a $\mathbb{C}$-scheme $B$ to the 
groupoid whose objects consist of data,
\begin{align}\label{C}
(\pi\colon \cC \to B, p_1, \cdots, p_m, \oO_{\cC}^{\oplus n} 
\stackrel{q}{\twoheadrightarrow} \qQ),
\end{align}
satisfying the following. }
\begin{itemize}
\item \emph{$(\pi \colon \cC \to B, p_1, \cdots, p_m)$
is a family of genus $g$, $m$-pointed quasi-stable curve over $B$. }
\item \emph{$\qQ$ is flat over $B$ such that for any $b\in B$, 
the data 
\begin{align*}
(\cC_b, p_1(b), \cdots, p_m(b), \oO_{\cC_b}^{\oplus n} 
\stackrel{q_b}{\twoheadrightarrow} \qQ_b),
\end{align*}
is an $\epsilon$-stable quotient of type $(r, n, d)$.}
\end{itemize}
\emph{For another object over $B$,
\begin{align}\label{Can}
(\pi'\colon \cC' \to B, p_1', \cdots, p_m', \oO_{\cC'}^{\oplus n} 
\stackrel{q'}{\twoheadrightarrow} \qQ'),
\end{align}
the set of isomorphisms between (\ref{C}) and (\ref{Can}) is given by 
\begin{align*}
\{ \phi \in 
\Isom_{\mM_{g, m}(B)}((\cC, p_1, \cdots, p_m), 
(\cC', p_1', \cdots, p_m')) \colon \Ker(q)=\Ker(\phi^{\ast}(q'))\}.
\end{align*}}
\end{defi}
By the construction, there is an obvious forgetting 
1-morphism, 
\begin{align}\label{forget}
\overline{\qQ}_{g, m}^{\epsilon}(\mathbb{G}(r, n), d)
\to \mM_{g, m}.
\end{align}
The following lemma shows that 
the automorphism groups in 
$\overline{\qQ}_{g, m}^{\epsilon}(\mathbb{G}(r, n), d)$
are finite. 
\begin{lem}\label{lem:aut}
For a genus $g$, $m$-pointed $\epsilon$-stable quotient  
$(\oO_C^{\oplus n} \stackrel{q}{\twoheadrightarrow}  Q)$
of type $(r, n, d)$, 
we have 
\begin{align}\label{aut:fin}
\sharp \Aut(\oO_C^{\oplus n} \stackrel{q}{\twoheadrightarrow} Q) <\infty, 
\end{align}
in the groupoid 
$\overline{\qQ}_{g, m}^{\epsilon}(\mathbb{G}(r, n), d)(\Spec \mathbb{C})$. 
\end{lem}
\begin{proof}
It is enough to show that for each irreducible 
component $P\subset C$, we have 
\begin{align*}
\sharp \Aut(\oO_P^{\oplus n} \stackrel{q|_{P}}{\twoheadrightarrow}
 Q|_{P})<\infty.
\end{align*}
Hence we may assume that $C$ is irreducible. 
The cases we need to consider are the following, 
\begin{align*}
(s(C), g(C))=(0, 0), (0, 1), (1, 0), (2, 0). 
\end{align*}
Here we have used the notation in Lemma~\ref{lem:ob}. 
For simplicity we treat the case of $(s(C), g(C))=(1, 0)$. 
The other cases are similarly discussed. 

Let $f$ be a rational map, 
\begin{align*}
f\colon C \dashrightarrow \mathbb{G}(r, n), 
\end{align*}
determined by the quotient $q$. 
(cf.~Remark~\ref{e1}.)
If $f$ is non-constant, then (\ref{aut:fin})
is obviously satisfied. Hence we may assume that $f$ is a 
constant rational map. 
By the equality (\ref{eq:deg}), 
this implies that the torsion subsheaf $\tau(Q) \subset Q$
satisfies 
\begin{align*}
\length \tau(Q)=\deg Q. 
\end{align*} 
Also if $\sharp \Supp \tau(Q) \ge 2$, then 
(\ref{aut:fin}) is satisfied, since 
any automorphism preserves torsion points and special 
points. Hence we may assume that there is 
unique $p\in C$ such that 
\begin{align*}
\length \tau(Q)_{p}=\length \tau(Q)=\deg Q. 
\end{align*}
However this contradicts to the condition (\ref{ineq})
and Lemma~\ref{lem:ob}. 
\end{proof}
We will show the following theorem. 
\begin{thm}\label{thm:rep}
The 2-functor 
$\overline{\qQ}_{g, m}^{\epsilon}(\mathbb{G}(r, n), d)$
is a proper Deligne-Mumford stack
of finite type over $\mathbb{C}$
with a perfect obstruction theory. 
\end{thm}
\begin{proof}
The construction of the moduli space 
and the properness will be postponed in Section~\ref{sec:proof}. 
The existence of the perfect obstruction theory 
will be discussed in Theorem~\ref{thm:vir}. 
\end{proof}

By theorem~\ref{thm:rep},
the 2-functor (\ref{2func})
is interpreted as a geometric object, 
rather than an abstract 2-functor.
In order to emphasize this,  
we slightly change the notation as follows. 
\begin{defi}
\emph{We denote the Deligne-Mumford
moduli stack of genus $g$, 
$m$-pointed $\epsilon$-stable quotients of 
type $(r, n, d)$ by 
\begin{align*}
\overline{Q}_{g, m}^{\epsilon}(\mathbb{G}(r, n), d).
\end{align*}}
\end{defi}
When $r=1$, we occasionally write
\begin{align*}
\overline{Q}_{g, m}^{\epsilon}(\mathbb{P}^{n-1}, d)
\cneq \overline{Q}_{g, m}^{\epsilon}(\mathbb{P}^{n-1}, d).
\end{align*}
The universal curve is denoted by 
\begin{align}\label{univ1}
\pi^{\epsilon} \colon U^{\epsilon}
\to \overline{Q}_{g, m}^{\epsilon}(\mathbb{G}(r, n), d),
\end{align}
and we have the universal quotient,
\begin{align}\label{univ2}
0\to S_{U^{\epsilon}} \to \oO_{U^{\epsilon}}^{\oplus n}
\stackrel{q_{U^{\epsilon}}}{\to} Q_{U^{\epsilon}} \to 0. 
\end{align}
\subsection{Structures of the moduli spaces of $\epsilon$-stable 
quotients}
\label{subsec:Pro}
 Below we discuss some structures on
 the moduli spaces of $\epsilon$-stable 
 quotients. Similar structures for MOP-stable 
 quotients are discussed in~\cite[Section~3]{MOP}.
 
Let $\overline{M}_{g, m}$ be the 
moduli stack of genus $g$, $m$-pointed 
stable curves. 
By composing (\ref{forget}) with the stabilization 
morphism, 
we obtain the proper morphism between Deligne-Mumford 
stacks, 
\begin{align*}
\nu^{\epsilon} \colon 
\overline{Q}_{g, m}^{\epsilon}(\mathbb{G}(r, n), d)
\to \overline{M}_{g, m}. 
\end{align*}
For an $\epsilon$-stable quotient
$\oO_C^{\oplus n} \twoheadrightarrow Q$
with markings $p_1, \cdots, p_m$,  
the sheaf $Q$ is locally free 
at $p_i$. Hence it determines an evaluation map, 
\begin{align}\label{mor:ev}
\ev_i \colon 
\overline{Q}_{g, m}^{\epsilon}(\mathbb{G}(r, n), d)
\to \mathbb{G}(r, n). 
\end{align}
Taking the fiber product, 
\begin{align}\label{fib:pro}
\xymatrix{
\overline{Q}_{g_1, m_1+1}^{\epsilon}(\mathbb{G}(r, n), d_1)
\times_{\ev}\overline{Q}_{g_2, m_2+1}^{\epsilon}(\mathbb{G}(r, n), d_2)
\ar[r]\ar[d] & \overline{Q}_{g_1, m_1+1}^{\epsilon}(\mathbb{G}(r, n), d_1)
\ar[d]^{\ev_{m_1+1}}, \\
\overline{Q}_{g_2, m_2+1}^{\epsilon}(\mathbb{G}(r, n), d_2)
\ar[r]^{\ev_1} & \mathbb{G}(r, n), }
\end{align}
we have the natural morphism, 
\begin{align}\notag
\overline{Q}_{g_1, m_1+1}^{\epsilon}(\mathbb{G}(r, n), d_1)
\times_{\ev}\overline{Q}_{g_2, m_2+1}^{\epsilon}&(\mathbb{G}(r, n), d_2) \\
\label{glue}
&\to \overline{Q}_{g_1+g_2, m_1+m_2}^{\epsilon}(\mathbb{G}(r, n), d_1+d_2),
\end{align}
defined by gluing $\epsilon$-stable quotients 
at the marked points. 
The standard $\GL_{n}(\mathbb{C})$-action 
on $\oO_C^{\oplus n}$ induces an 
$\GL_n(\mathbb{C})$-action on 
$\overline{Q}_{g, m}^{\epsilon}(\mathbb{G}(r, n), d)$, 
i.e. 
\begin{align*}
g\cdot (\oO_C^{\oplus n}\stackrel{q}{\twoheadrightarrow}
Q)=(\oO_C^{\oplus n}\stackrel{q\circ g}{\twoheadrightarrow} Q),
\end{align*}
for $g\in \GL_n(\mathbb{C})$. 
The morphisms (\ref{mor:ev}), (\ref{glue})
are $\GL_n(\mathbb{C})$-equivariant. 
\subsection{Virtual fundamental classes}
The moduli space of $\epsilon$-stable quotients have 
the associated virtual fundamental class. 
The following is an analogue of \cite[Theorem~2, Lemma~4]{MOP}
in our situation.  
\begin{thm}\label{thm:vir}
There is a $\GL_n(\mathbb{C})$-equivariant
 $2$-term perfect obstruction theory 
on $\overline{Q}_{g, m}^{\epsilon}(\mathbb{G}(r, n), d)$. 
In particular there is a 
virtual fundamental class, 
\begin{align*}
[\overline{Q}_{g, m}^{\epsilon}(\mathbb{G}(r, n), d)]^{\rm vir}
\in A_{\ast}^{\GL_n(\mathbb{C})}
(\overline{Q}_{g, m}^{\epsilon}(\mathbb{G}(r, n), d), \mathbb{Q}),
\end{align*}
in the $\GL_n(\mathbb{C})$-equivariant Chow group. 
The virtual dimension is given by 
\begin{align*}
nd +r(n-r)(1-g)+3g-3+m,
\end{align*}
which does not depend on a choice of $\epsilon$. 
\end{thm}
\begin{proof}
The same argument of~\cite[Theorem~2, Lemma~4]{MOP}
works. For the reader's convenience, we provide 
the argument. 
For a fixed marked quasi-stable curve, 
\begin{align*}
(C, p_1, \cdots, p_m)\in \mM_{g, m}, 
\end{align*}
the moduli space of $\epsilon$-stable quotients
is an open set of the Quot scheme. 
On the other hand, the deformation theory 
of the Quot scheme
on a non-singular curve
is obtained in
~\cite{CK}, \cite{MO}. 
Noting that any quasi-stable quotient 
is locally free near nodes, the analogues construction 
yields the 2-term 
obstruction theory relative to the forgetting 
1-morphism $\nu$, 
\begin{align*}
\nu \colon \overline{Q}_{g, m}^{\epsilon}(\mathbb{G}(r, n), d)
\to \mM_{g, m}, 
\end{align*}
given by $\dR \pi^{\epsilon}_{\ast}
 \hH om(S_{U^{\epsilon}}, Q_{U^{\epsilon}})^{\ast}$. 
(See~(\ref{univ1}), (\ref{univ2}).)
The absolute obstruction theory is given by the cone
$E^{\bullet}$ of 
the morphism~\cite{BF}, \cite{GP},  
\begin{align*}
\dR \pi^{\epsilon}_{\ast}
 \hH om(S_{U^{\epsilon}}, Q_{U^{\epsilon}})^{\ast}
 \to \nu^{\ast}\mathbb{L}_{\mM_{g, m}}[1], 
\end{align*}
where $\mathbb{L}_{\mM_{g, m}}$ is the 
cotangent complex of the algebraic stack $\mM_{g, m}$. 
By Lemma~\ref{lem:aut}, the complex $E^{\bullet}$ 
is concentrated on $[-1, 0]$. 
Let $\oO_C^{\oplus n} \twoheadrightarrow Q$ be 
an $\epsilon$-stable quotient with kernel $S$ and 
marked points $p_1, \cdots, p_m$. 
By the above description of 
the obstruction theory and the Riemann-Roch theorem, the virtual 
dimension is given by 
\begin{align*}
\chi(S, Q)-\chi(T_{C}(-\sum_{i=1}^m p_i)) 
=nd+r(n-r)(1-g)+3g-3+m. 
\end{align*}
\end{proof}
By the proof of the above theorem, 
the tangent space $\Tan_{q}$ and 
the obstruction space $\Obs_{q}$
at the $\epsilon$-stable quotient 
$q\colon \oO_C^{\oplus n}\twoheadrightarrow Q$
with kernel $S$ and marked points $p_1, \cdots, p_m$
fit into the exact sequence, 
\begin{align}\notag
0 & \to H^0(C, T_C(-\sum_{i=1}^{m}p_i)) \to 
\Hom(S, Q) \to \Tan_{q} \\
\label{tanob}
& \to H^1(C, T_C(-\sum_{i=1}^{m}p_i)) 
\to \Ext^1(S, Q) \to \Obs_{q}\to 0. 
\end{align}
In the genus zero case, the obstruction space
vanishes hence the moduli space is non-singular. 
\begin{lem}\label{lem:nonsing}
The Deligne-Mumford stack
$\overline{Q}_{0, m}^{\epsilon}(\mathbb{G}(r, n), d)$
is non-singular of expected dimension 
$nd+r(n-r)+m-3$. 
\end{lem}
\begin{proof}
In the notation of the exact sequence (\ref{tanob}), 
it is enough to see that 
\begin{align}\label{genus0}
\Ext^1(S, Q)=H^0(C, S\otimes \widetilde{Q}^{\vee}\otimes \omega_C)^{\ast}=0,
\end{align}
when the genus of $C$ is zero. 
Here $\widetilde{Q}$ is the free part of $Q$,
i.e. $Q/\tau(Q)$ for the torsion subsheaf $\tau(Q)\subset Q$. 
For any irreducible component $P\subset C$
with $s(P)=1$, 
it is easy to see that 
\begin{align*}
\deg(S\otimes \widetilde{Q}^{\ast}\otimes \omega_C)|_{P}<0,
\end{align*}
by Lemma~\ref{lem:ob}. 
Then it is easy to deduce the vanishing (\ref{genus0}). 
\end{proof}

\subsection{Wall-crossing phenomena of $\epsilon$-stable 
quotients}
Here we see that there is a finite number of 
values in $\mathbb{R}_{>0}$ 
so that the moduli spaces of $\epsilon$-stable 
quotients are constant on each interval. 
First we treat the case of 
\begin{align}\label{gm}
(g, m)\neq (0, 0). 
\end{align}
We set 
\begin{align*}
0=\epsilon_{0}<\epsilon_1 <\cdots <\epsilon_d<\epsilon_{d+1}=\infty,
\end{align*}
as follows, 
\begin{align}\label{ei}
\epsilon_{i}=\frac{1}{d-i+1}, \quad 1\le i\le d. 
\end{align}
\begin{prop}\label{prop:wall}
Under the condition (\ref{gm}),  
take $\epsilon \in (\epsilon_{i-1}, \epsilon_i]$
where $\epsilon_i$ is given by (\ref{ei}). 
Then we have 
\begin{align*}
\overline{Q}_{g, m}^{\epsilon}(\mathbb{G}(r, n), d)=
\overline{Q}_{g, m}^{\epsilon_i}(\mathbb{G}(r, n), d).
\end{align*}
\end{prop}
\begin{proof}
Let us take a quasi-stable quotient of type $(r, n, d)$, 
\begin{align}\label{takest}
(\oO_C^{\oplus n} \stackrel{q}{\twoheadrightarrow} Q).
\end{align}
First we show that if (\ref{takest}) is $\epsilon$-stable, 
then it is also $\epsilon_i$-stable. 
Since $\epsilon\le \epsilon_i$, the ampleness 
of $\lL(q, \epsilon)$ also implies the ampleness of 
$\lL(q, \epsilon_i)$. 
For $p\in C$, let us denote by $l_p$ the length 
of $\tau(Q)$ at $p$. 
If $l_p\neq 0$, the condition (\ref{ineq}) implies 
\begin{align}\label{ineq2}
0<\epsilon \le \frac{1}{l_p}. 
\end{align}
Since $l_p \le d$, 
the inequality (\ref{ineq2})
also implies $\epsilon_i \le 1/l_p$, which 
in turn implies the condition (\ref{ineq}) for $\epsilon_i$. 

Conversely suppose that the quasi-stable quotient 
(\ref{takest})
is $\epsilon_i$-stable. 
The inequality (\ref{ineq}) for $\epsilon_i$ 
also implies (\ref{ineq}) for $\epsilon$
since $\epsilon \le \epsilon_i$.  
In order to see that $\lL(q, \epsilon)$
is ample, take an irreducible component 
$P\subset C$ and 
check (\ref{ob1}), (\ref{ob2}) and (\ref{ob3}). 
The condition (\ref{ob1}) does not depend on $\epsilon$, 
so (\ref{ob1}) is satisfied. Also the assumption (\ref{gm})
implies that the case (\ref{ob3}) does not occur, 
hence we only have to check (\ref{ob2}). 
We denote by $d_P$ the degree of $Q|_{P}$. 
If $s(P)=1$ and $g(P)=0$, we have 
\begin{align}\label{ineqd}
\epsilon_i>\frac{1}{d_P}, 
\end{align}
by the condition (\ref{ob2}) for $\epsilon_i$. 
Since $d_P\le d$, (\ref{ineqd}) implies 
\begin{align*}
\epsilon>\epsilon_{i-1} \ge \frac{1}{d_P}, 
\end{align*}
which in turn implies the condition (\ref{ob2}) 
for $\epsilon$. Hence (\ref{takest})
is $\epsilon$-stable. 
\end{proof}
Next we treat the case of $(g, m)=(0, 0)$. 
In this case, the moduli space 
is empty for small $\epsilon$. 
\begin{lem}\label{lem:empty}
For $0<\epsilon \le 2/d$, we have 
\begin{align*}
\overline{Q}_{0, 0}^{\epsilon}(\mathbb{G}(r, n), d)=\emptyset. 
\end{align*}
\end{lem}
\begin{proof}
If $\overline{Q}_{g, m}^{\epsilon}(\mathbb{G}(r, n), d)$
is non-empty, the ampleness of $\lL(q, \epsilon)$ yields, 
\begin{align*}
2g-2+m+\epsilon \cdot d>0. 
\end{align*}
Hence if $g=m=0$, the $\epsilon$ should 
satisfy $\epsilon >2/d$. 
\end{proof}
Let $d'\in \mathbb{Z}$ be the integer part of 
$d/2$.  
We set $0=\epsilon_0<\epsilon_1<\cdots$ in the 
following way. 
\begin{align}\notag
&\epsilon_1=2, \ \epsilon_2=\infty, \quad (d=1) \\
\label{d:odd}
&\epsilon_1=\frac{2}{d}, \
\epsilon_i=\frac{1}{d'-i+2}, \ (2\le i\le d'+1), \
\epsilon_{d'+2}=\infty, \quad ( d\ge 3 \mbox{ is odd.}) \\
\label{d:even}
& \epsilon_i=\frac{1}{d'-i+1}, \ (1\le i\le d'), \ 
\epsilon_{d'+1}=\infty, \quad ( d \mbox{ is even.}) 
\end{align}
We have the following. 
\begin{prop}\label{wall2}
For $\epsilon_{\bullet}$ as above, we have 
\begin{align*}
\overline{Q}_{0, 0}^{\epsilon}(\mathbb{G}(r, n), d)=
\overline{Q}_{0, 0}^{\epsilon_i}(\mathbb{G}(r, n), d),
\end{align*}
for $\epsilon \in (\epsilon_{i-1}, \epsilon_i]$.
\end{prop}
\begin{proof}
By Lemma~\ref{lem:empty}, 
we may assume that $\epsilon_{i-1}\ge 2/d$.
Then we can follow the essentially same argument of
Proposition~\ref{prop:wall}.   
The argument is more subtle since we have to take 
 the condition (\ref{ob3}) into consideration,
  but we leave the detail to the reader. 
\end{proof}
Let $\overline{M}_{g, m}(\mathbb{G}(r, n), d)$ 
be the moduli space of genus $g$, $m$-pointed stable maps
$f\colon C \to \mathbb{G}(r, n)$, satisfying 
\begin{align*}
f_{\ast}[C]=d \in H_2(\mathbb{G}(r, n), \mathbb{Z})\cong \mathbb{Z}. 
\end{align*}
(cf.~\cite{Ktor}.)
Also we denote by $\overline{Q}_{g, m}(\mathbb{G}(r, n), d)$
the moduli space of MOP-stable quotients
of type $(r, n, d)$,  
constructed in~\cite{MOP}. 
By the following result, we see that 
both moduli spaces are related by 
wall-crossing phenomena
of $\epsilon$-stable quotients. 
\begin{thm}\label{thm:cross}
(i) For $\epsilon>2$, we have 
\begin{align}\label{isom1}
\overline{Q}_{g, m}^{\epsilon}(\mathbb{G}(r, n), d)
\cong 
\overline{M}_{g, m}(\mathbb{G}(r, n), d).
\end{align}
(ii) For $0<\epsilon \le 1/d$, we have 
\begin{align}
\label{isom2}
\overline{Q}_{g, m}^{\epsilon}(\mathbb{G}(r, n), d)
\cong 
\overline{Q}_{g, m}(\mathbb{G}(r, n), d). 
\end{align}
\end{thm}
\begin{proof}
(i) First take an $\epsilon$-stable 
quotient $\oO_C^{\oplus n}\stackrel{q}{\twoheadrightarrow} Q$
for some $\epsilon>2$, 
with marked points $p_1, \cdots, p_m$.
 By Proposition~\ref{prop:wall}
and Proposition~\ref{wall2}, we may take 
$\epsilon=3$. 
The condition (\ref{ineq})
implies that $Q$ is locally free, hence 
$q$ determines a map, 
\begin{align}\label{stabmap}
f\colon C \to \mathbb{G}(r, n). 
\end{align}
Also the ampleness of $\lL(q, 3)$ is equivalent to the 
ampleness of 
the line bundle 
\begin{align}\label{amp:imp}
\omega_C(p_1+\cdots p_m)\otimes f^{\ast}O_{G}(3),
\end{align}
where $\oO_G(1)$ is the 
restriction of $\oO(1)$
to $\mathbb{G}(r, n)$ via the Pl$\ddot{\rm{u}}$cker
embedding. 
The ampleness of (\ref{amp:imp}) implies 
that the map $f$ is a stable map. 

Conversely take an $m$-pointed stable map,
\begin{align*}
f\colon C\to \mathbb{G}(r, n), \quad 
p_1, \cdots, p_m \in C, 
\end{align*} 
and a quotient $\oO_C^{\oplus n} \stackrel{q}{\twoheadrightarrow} Q$
by pulling back the universal quotient on $\mathbb{G}(r, n)$
via $f$. Then the stability of the map $f$
implies the ampleness of the line bundle (\ref{amp:imp}), 
hence the ampleness of $\lL(q, 3)$. 
Also the condition (\ref{ineq}) is automatically 
satisfied for $\epsilon=3$ since $Q$ is locally free.
Hence we obtain the isomorphism (\ref{isom1}). 

(ii)  
If $(g, m)=(0, 0)$, then both sides of (\ref{isom2}) are empty,
so we may assume that $(g, m)\neq (0, 0)$. 
Let us take an $\epsilon$-stable quotient 
$\oO_C^{\oplus n}\stackrel{q}{\twoheadrightarrow} Q$
for $0<\epsilon \le 1/d$. 
For any irreducible component $P\subset C$, we have 
$\deg(Q|_{P})\le d$. By Lemma~\ref{lem:ob}, this implies that 
there is no irreducible component $P\subset C$
with 
\begin{align*}
(s(P), g(P))=(0, 0) \mbox{ or }(0, 1).
\end{align*}
Hence applying Lemma~\ref{lem:ob} 
again, we see that $q$ is MOP-stable. 

Conversely take a
MOP-stable quotient 
$\oO_C^{\oplus n}\stackrel{q}{\twoheadrightarrow} Q$
and $0<\epsilon \le 1/d$. By the definition of MOP-stable quotient, 
the line bundle $\lL(q, \epsilon)$ is ample. 
Also for any point $p\in C$, the length of the torsion part of 
$Q$ is less than or equal to $d$. 
(cf.~Remark~\ref{e1}). 
Hence the condition (\ref{ineq}) is satisfied and $q$ is 
$\epsilon$-stable. Therefore the desired isomorphism (\ref{isom2}) 
holds. 
\end{proof}
\subsection{Morphisms between moduli spaces of $\epsilon$-stable 
quotients}\label{subsec:Mor}
In this subsection, 
we construct some natural morphisms 
between moduli spaces of $\epsilon$-stable quotients. 
The first one is an analogue of the
Pl$\ddot{\rm{u}}$cker embedding.
(See~\cite[Section~5]{MOP} for the corresponding
morphism between MOP-stable quotients.)
\begin{lem}
There is a natural morphism, 
\begin{align}\label{iota}
\iota^{\epsilon} \colon 
\overline{Q}_{g, m}^{\epsilon}(\mathbb{G}(r, n), d) 
\to \overline{Q}_{g, m}^{\epsilon}(
\mathbb{G}(1, \dbinom{n}{r}), d).
\end{align}
\end{lem}
\begin{proof}
For a quasi-stable quotient $\oO_C^{\oplus n}\stackrel{q}{\twoheadrightarrow}
Q$ 
of type $(r, n, d)$ 
with kernel $S$, we associate the exact sequence, 
\begin{align*}
0\to \wedge^{r} S \to \wedge^{r}\oO_C^{\oplus n}
\stackrel{q'}{\to}Q' \to 0.  
\end{align*}
It is easy to see that $q$ is $\epsilon$-stable 
if and only if $q'$ is $\epsilon$-stable. 
The map $q\mapsto q'$ gives the desired morphism. 
\end{proof}
Next we treat the case of $r=1$. 
\begin{prop}\label{r=1}
For $\epsilon \ge \epsilon'$, there is a 
natural morphism, 
\begin{align}\label{mor:c}
c_{\epsilon, \epsilon'} \colon 
\overline{Q}_{g, m}^{\epsilon}(\mathbb{P}^{n-1}, d) 
\to \overline{Q}_{g, m}^{\epsilon'}(\mathbb{P}^{n-1}, d). 
\end{align}
\end{prop}
\begin{proof}
For simplicity we deal with the case of $(g, m)\neq (0, 0)$. 
By Proposition~\ref{prop:wall}, it is 
enough to construct a morphism 
\begin{align}\label{ci}
c_{i+1, i}\colon 
\overline{Q}_{g, m}^{\epsilon_{i+1}}(\mathbb{P}^{n-1}, d) 
\to \overline{Q}_{g, m}^{\epsilon_i}(\mathbb{P}^{n-1}, d),
\end{align}
where $\epsilon_i$ is given by (\ref{ei}). 
Let us take an $\epsilon_{i+1}$-stable 
quotient $\oO_C^{\oplus n}\stackrel{q}{\twoheadrightarrow}Q$, 
and the set of irreducible components $T_1, \cdots, T_k$ 
of $C$ satisfying 
\begin{align}\label{T:sat}
(s(T_j), g(T_j))=(1, 0), \quad \deg(Q|_{T_j})=d-i+1.
\end{align}
Note that $T_j$ and $T_{j'}$ are disjoint 
for $j\neq j'$, by the 
assumption $(g, m)\neq (0, 0)$. 
We set $T$ and $C'$ to be 
\begin{align}\label{TC'}
T=\amalg_{j=1}^{k}T_j, \quad C'=\overline{C\setminus T}. 
\end{align}
The intersection $T_j \cap C'$ consists of 
one point $x_j$, unless $(g, m)=(0, 1)$, $k=1$ and $i=1$. 
In the latter case, the space 
$\overline{Q}_{0, 1}^{\epsilon_1}(\mathbb{P}^{n-1}, d)$
is empty, so there is nothing to prove. 
Let $S$ be the kernel of $q$. We have the sequence of 
inclusions, 
\begin{align*}
S'\cneq S|_{C'}(-\sum_{j=1}^k (d-i+1)x_j) \hookrightarrow 
S|_{C'} \hookrightarrow \oO_{C'}^{\oplus n}, 
\end{align*}
and the exact sequence, 
\begin{align*}
0 \to S' \to \oO_{C'}^{\oplus n} \stackrel{q'}{\to}Q' \to 0.
\end{align*}
It is easy to see that $q'$ is an $\epsilon_i$-stable
quotient. Then the map 
$q\mapsto q'$ gives the desired morphism (\ref{ci}). 
\end{proof}
 \begin{rmk}
 Suppose that $(g, m)\neq (0, 0)$. 
 By Proposition~\ref{r=1}, Proposition~\ref{prop:wall}
 and Theorem~\ref{thm:cross}, 
 we have the sequence of morphisms, 
  \begin{align}\notag
 \overline{M}_{g, m}(\mathbb{P}^{n-1}, d) =
 \overline{Q}_{g, m}^{\epsilon_{d+1}}(\mathbb{P}^{n-1}, d)
 \to \overline{Q}_{g, m}^{\epsilon_{d}}(\mathbb{P}^{n-1}, d)
 \to \cdots \\ \label{mor:seq}
 \cdots \to 
 \overline{Q}_{g, m}^{\epsilon_{2}}(\mathbb{P}^{n-1}, d)
 \to \overline{Q}_{g, m}^{\epsilon_{1}}(\mathbb{P}^{n-1}, d)
 =\overline{Q}_{g, m}(\mathbb{P}^{n-1}, d). 
 \end{align}
 The composition of the above morphism 
 \begin{align}\label{mor:cc}
 c\colon \overline{M}_{g, m}(\mathbb{P}^{n-1}, d) \to 
 \overline{Q}_{g, m}(\mathbb{P}^{n-1}, d)
 \end{align}
 coincides with the morphism 
 constructed in~\cite[Section~5]{MOP}. 
 The morphism $c$ 
 also appears for the Quot scheme of a fixed 
 non-singular curve in~\cite{PR}. 
  \end{rmk}
  Let us investigate the morphism (\ref{ci}) 
  more precisely. 
 For $k\in \mathbb{Z}_{\ge 0}$, 
 we consider a subspace,
 \begin{align}\label{sub:k}
 \overline{Q}_{g, m}^{\epsilon_{i+1}, k+}(\mathbb{P}^{n-1}, d)
 \subset \overline{Q}_{g, m}^{\epsilon_{i+1}}(\mathbb{P}^{n-1}, d),
 \end{align}
 consisting of $\epsilon_{i+1}$-stable 
 quotients with
   exactly $k$-irreducible 
  components $T_1, \cdots, T_k$ 
  satisfying (\ref{T:sat}). 
  Setting $d_i=d-i+1$, the subspace (\ref{sub:k})
fits into the Cartesian diagram, 
\begin{align}\label{Car}
\xymatrix{
  \overline{Q}_{g, m}^{\epsilon_{i+1}, k+}(\mathbb{P}^{n-1}, d)
 \ar[r]\ar[d] &
  \overline{Q}_{0, 1}^{\epsilon_{i+1}}(\mathbb{P}^{n-1}, d_i)^{\times k}
  \ar[d]^{(\ev_1)^{\times k}}, \\
  \overline{Q}_{g, k+m}^{\epsilon_{i}, \epsilon_{i+1}}
(\mathbb{P}^{n-1}, d-kd_i)
 \ar[r] &
(\mathbb{P}^{n-1})^{\times k}. }
\end{align}
Here the bottom arrow is the evaluation map 
with respect to the first $k$-marked points, and  
   the space 
   \begin{align}\label{both}
   \overline{Q}_{g, m}^{\epsilon_{i}, \epsilon_{i+1}}
  (\mathbb{P}^{n-1}, d)
  \end{align}
  is the moduli space of genus $g$, $m$-marked quasi-stable 
  quotients of type $(1, n, d)$, which 
  is both $\epsilon_i$ and $\epsilon_{i+1}$-stable. 
  The space (\ref{both}) is an open Deligne-Mumford
  substack of $\overline{Q}_{g, m}^{\epsilon}
  (\mathbb{P}^{n-1}, d)$
  for both $\epsilon=\epsilon_{i}$ and $\epsilon_{i+1}$. 
Note that the left arrow of (\ref{Car}) is surjective 
since the right arrow is surjective. 
    
 We also consider a subspace 
  \begin{align*}
  \overline{Q}_{g, m}^{\epsilon_{i}, k-}(\mathbb{P}^{n-1}, d)
 \subset \overline{Q}_{g, m}^{\epsilon_{i}}(\mathbb{P}^{n-1}, d),
  \end{align*}
  consisting of $\epsilon_i$-stable quotients 
  $\oO_C^{\oplus n}\stackrel{q}{\twoheadrightarrow} Q$
  with exactly $k$-distinct points $x_1, \cdots, x_k \in C$ 
  satisfying 
  \begin{align*}
   \length \tau(Q)_{x_j}=d_i, \quad 
  1\le j\le k. 
  \end{align*} 
  Obviously we have the isomorphism,
  \begin{align}\label{isom:ob}
  \overline{Q}_{g, m}^{\epsilon_{i}, k-}(\mathbb{P}^{n-1}, d)
  \cong  
  \overline{Q}_{g, k+m}^{\epsilon_{i}, \epsilon_{i+1}}
  (\mathbb{P}^{n-1}, d-kd_i),
  \end{align}
  and the construction of (\ref{ci}) yields the Cartesian 
  diagram,
  \begin{align}\label{dig:Car}
  \xymatrix{
  \overline{Q}_{g, m}^{\epsilon_{i+1}, k+}(\mathbb{P}^{n-1}, d)
 \ar[r]\ar[d] &
  \overline{Q}_{g, m}^{\epsilon_{i+1}}(\mathbb{P}^{n-1}, d) 
  \ar[d]^{c_{i+1, i}}, \\
  \overline{Q}_{g, m}^{\epsilon_{i}, k-}(\mathbb{P}^{n-1}, d)
 \ar[r] &
  \overline{Q}_{g, m}^{\epsilon_{i}}(\mathbb{P}^{n-1}, d). }
  \end{align}
  The left arrow of the diagram (\ref{dig:Car})
 coincides with the left arrow of (\ref{Car})  
  under the isomorphism (\ref{isom:ob}), 
and in particular it is surjective. 
The above argument implies the following. 
\begin{lem}\label{lem:surj}
The morphism $c_{\epsilon, \epsilon'}$ constructed in Proposition~\ref{r=1} 
is surjective. 
\end{lem}
  
  For $r>1$, 
it seems that there is no natural 
morphism between 
$\overline{M}_{g, m}(\mathbb{G}(r, n), d)$
and $\overline{Q}_{g, m}(\mathbb{G}(r, n), d)$,
as pointed out in~\cite{MOP}, \cite{PR}. 
However for $\epsilon=1$, there is a natural 
morphism between moduli spaces of stable maps 
and those of $\epsilon$-stable quotients. 
The following lemma will be used in Lemma~\ref{lem:eqconn} below. 
\begin{lem}\label{lem:worth}
There is a natural surjective morphism, 
\begin{align*}
c'\colon 
\overline{M}_{g, m}(\mathbb{G}(r, n), d)
\to \overline{Q}_{g, m}^{\epsilon=1}(\mathbb{G}(r, n), d).
\end{align*}
\end{lem} 
\begin{proof}
For simplicity, we assume that $(g, m)\neq (0, 0)$. 
For a stable map $f\colon C \to \mathbb{G}(r, n)$
of degree $d$, pulling back the universal quotient yields
the exact sequence, 
\begin{align}\label{asseq}
0 \to S \to \oO_C^{\oplus n} \stackrel{q}{\to}Q \to 0. 
\end{align}
Here $Q$ is a locally free sheaf on $C$ and 
the quotient $q$ is of type $(r, n, d)$. 
Let $T_1, \cdots, T_k$ be the set of irreducible 
components of $C$, satisfying the following, 
\begin{align*}
(s(T_j), g(T_j))=(1, 0), \quad \deg(Q|_{T_j})=1.
\end{align*}
By the exact sequence (\ref{asseq}) 
and the degree reason, the following isomorphisms exist, 
\begin{align}\label{isomS}
Q|_{T_j}\cong \oO_{\mathbb{P}^1}(1)\oplus \oO_{\mathbb{P}^1}^{\oplus 
n-r-1}, \quad 
S|_{T_j}\cong \oO_{\mathbb{P}^1}(-1)\oplus \oO_{\mathbb{P}^1}^{\oplus r-1}.
\end{align}
We set $T$ and $C'$ as in (\ref{TC'}), 
and set $x_j=T_j \cap C'$. 
Let $\pi$ be the morphism 
\begin{align*}
\pi \colon C \to C', 
\end{align*}
which is identity outside $T$ and contracts
$T_j$ to $x_j$. 
The exact sequences 
\begin{align}
0 \to Q|_{T}(-\sum_{j=1}^{k}x_j) \to Q \to Q|_{C'} \to 0, \\
0 \to S|_{C'}(-\sum_{j=1}^{k}x_j) \to S \to S|_{T} \to 0,
\end{align}
and the isomorphisms (\ref{isomS}) show that 
$\pi_{\ast}Q$ has torsion at $x_j$ with length one and
$R^1 \pi_{\ast}S=0$. Therefore
applying $\pi_{\ast}$ to (\ref{asseq})
yields the
exact sequence, 
\begin{align*}
0 \to \pi_{\ast}S \to \oO_{C'}^{\oplus n} \stackrel{q'}{\to}
 \pi_{\ast}Q \to 0. 
\end{align*}
It is easy to see that $q'$ is an $\epsilon$-stable 
quotient with $\epsilon=1$, and the map $f \mapsto q'$ 
gives the desired morphism $c'$. 
An argument similar to Lemma~\ref{lem:surj} shows that the 
morphism $c'$ is surjective. 
\end{proof}
\subsection{Wall-crossing formula 
of virtual fundamental classes}
In~\cite[Theorem~3, Theorem~4]{MOP}, 
the virtual fundamental classes
on moduli spaces of stable maps and 
those of MOP-stable quotients are compared. 
Such a comparison result also holds 
for $\epsilon$-stable quotients. 
Note that the arguments in Subsections~\ref{subsec:Pro}, \ref{subsec:Mor}
yield the following diagram:
\begin{align*}
\xymatrix{
& \ar[dl]_{\ev_i} \overline{Q}_{g, m}^{\epsilon}(\mathbb{G}(r, n), d) 
\ar[r]^{\iota^{\epsilon}} & 
\overline{Q}_{g, m}^{\epsilon}(\mathbb{G}(1,\dbinom{n}{r} ), d)
\ar[dd] _{c_{\epsilon, \epsilon'}}, \\
\mathbb{G}(r, n) & & \\
& \ar[ul]^{\ev_i} \overline{Q}_{g, m}^{\epsilon'}(\mathbb{G}(r, n), d) 
\ar[r]^{\iota^{\epsilon'}} & 
\overline{Q}_{g, m}^{\epsilon'}(\mathbb{G}(1, \dbinom{n}{r}), d).
}
\end{align*}
The following theorem, which 
is a refinement of~\cite[Theorem~4]{MOP}, 
is interpreted as a wall-crossing formula 
of GW type invariants. 
The proof will be given in Section~\ref{sec:cycle}. 
\begin{thm}\label{thm:wcf}
Take $\epsilon \ge \epsilon'>0$ satisfying 
$2g-2+\epsilon' \cdot d>0$. 
We have the formula, 
\begin{align}\label{WCF} 
c_{\epsilon, \epsilon' \ast}\iota^{\epsilon}_{\ast}
 [\overline{Q}_{g, m}^{\epsilon}(\mathbb{G}(r, n), d)]^{\rm vir}
= \iota^{\epsilon'}_{\ast}
 [\overline{Q}_{g, m}^{\epsilon'}(\mathbb{G}(r, n), d)]^{\rm vir}.
\end{align}
In particular for classes 
$\gamma_i 
\in A_{\GL_n(\mathbb{C})}^{\ast}(\mathbb{G}(r, n), \mathbb{Q})$, 
the following holds, 
\begin{align}\notag
& c_{\epsilon, \epsilon' \ast}\iota^{\epsilon}_{\ast}\left( 
\prod_{i=1}^{m}\ev_i^{\ast}(\gamma_i) \cap
 [\overline{Q}_{g, m}^{\epsilon}(\mathbb{G}(r, n), d)]^{\rm vir}
\right) 
= \\ \label{WCF2}
& \qquad \qquad \qquad \qquad \qquad 
\iota^{\epsilon'}_{\ast}\left( 
\prod_{i=1}^{m}\ev_i^{\ast}(\gamma_i) \cap 
 [\overline{Q}_{g, m}^{\epsilon'}(\mathbb{G}(r, n), d)]^{\rm vir}
\right).
\end{align}
\end{thm} 
\begin{rmk}
The formula (\ref{WCF}) 
in particular implies the formula, 
\begin{align}\label{WCF2}
c_{\ast}
 [\overline{Q}_{g, m}^{\epsilon}(\mathbb{P}^{n-1}, d)]^{\rm vir}
= [\overline{Q}_{g, m}^{\epsilon'}(\mathbb{P}^{n-1}, d)]^{\rm vir}.
\end{align}
Here the morphism $c$ is given by (\ref{mor:cc}). 
Applying the formula (\ref{WCF2})
to the diagram (\ref{mor:seq}) repeatedly, 
we obtain the following formula, 
\begin{align*}
c_{\ast}
 [\overline{M}_{g, m}(\mathbb{P}^{n-1}, d)]^{\rm vir}
= [\overline{Q}_{g, m}(\mathbb{P}^{n-1}, d)]^{\rm vir},
\end{align*}
which reconstructs the result of~\cite[Theorem~3]{MOP}, 
\end{rmk} 
\section{Type $(1, 1, d)$-quotients}\label{sec:ex}
In this section, we 
investigate the moduli spaces of 
 $\epsilon$-stable quotients of type $(1, 1, d)$ 
and relevant wall-crossing phenomena. 
\subsection{Relation to Hassett's weighted pointed 
stable curves}\label{subsec:Hassett}
Here we see that $\epsilon$-stable quotients
of type $(1, 1, d)$ 
are closely related to Hassett's weighted 
pointed stable curves~\cite{BH}, 
which we recall here. 
Let us take a sequence, 
\begin{align*}
a=(a_1, a_2, \cdots, a_m)\in (0, 1]^{m}. 
\end{align*}
\begin{defi}\label{def:weighted}\emph{
A data $(C, p_1, \cdots, p_m)$
of a
nodal curve $C$ and
(possibly not distinct) marked points $p_i \in C^{ns}$ is called
\textit{$a$-stable} if the following conditions 
hold. }
\begin{itemize}
\item \emph{The $\mathbb{R}$-divisor 
$K_C +\sum_{i=1}^{m}a_i p_i$ is ample. }
\item \emph{For any $p\in C$, we have 
$\sum_{p_i=p}a_i \le 1$. } 
\end{itemize}
\end{defi}
Note that setting $a_i=1$ for all $i$
yields the usual $m$-pointed stable curves. 
The moduli space of genus $g$, $m$-pointed 
$a$-stable curves is constructed in~\cite{BH} 
as a proper smooth Deligne-Mumford stack
over $\mathbb{C}$. 
Among weights, we only use the following 
weight for $\epsilon \in (0, 1]$, 
\begin{align}\label{amd}
a(m, d, \epsilon)\cneq (\displaystyle\overbrace{1, \cdots, 1}^{m}, 
\displaystyle\overbrace{\epsilon, \cdots, \epsilon}^{d}). 
\end{align}
The moduli space of genus $g$, 
$m+d$-pointed $a(m, d, \epsilon)$-stable 
curves is denoted by 
\begin{align}\label{denoted}
\overline{M}_{g, m|d}^{\epsilon}. 
\end{align}
If $m=0$, we simply write (\ref{denoted}) as 
$\overline{M}_{g, d}^{\epsilon}$. 
For $\epsilon \ge \epsilon'$, there is a natural 
birational contraction~\cite[Theorem~4.3]{BH}, 
\begin{align}\label{nat:bir}
c_{\epsilon, \epsilon'}\colon 
\overline{M}_{g, m|d}^{\epsilon} \to 
\overline{M}_{g, m|d}^{\epsilon'}. 
\end{align}
Now we describe the moduli spaces of $\epsilon$-stable 
quotients of type $(1, 1, d)$, and relevant 
wall-crossing phenomena.
In what follows, we denote by 
\begin{align*}
\pt \cneq \mathbb{P}^{0}=\mathbb{G}(1, 1) \cong \Spec \mathbb{C}. 
\end{align*}
We have the following proposition. (See~\cite[Proposition~3]{MOP} for the 
corresponding result of MOP-stable quotients.) 
\begin{prop}\label{prop:MOP}
We have the isomorphism, 
\begin{align}\label{isom:phi}
\phi \colon 
\overline{M}_{g, m|d}^{\epsilon}/S_d \stackrel{\sim}{\to}
\overline{Q}_{g, m}^{\epsilon}(\pt, d), 
\end{align}
where the symmetric group $S_d$ acts by permuting the
last $d$-marked points. 
\end{prop}
\begin{proof}
Take a genus $g$, $m+d$-pointed 
$a(m, d, \epsilon)$-stable curve, 
\begin{align*}
(C, p_1, \cdots, p_m, \widehat{p}_1, \cdots, \widehat{p}_d). 
\end{align*}
We associate the genus $g$, $m$-pointed quasi-stable 
quotient of type $(1, 1, d)$ by the 
exact sequence, 
\begin{align*}
0 \to \oO_C(-\sum_{j=1}^{d}\widehat{p}_j) \to \oO_C 
\stackrel{q}{\to} Q \to 0, 
\end{align*}
with $m$-marked points $p_1, \cdots, p_m$. 
The $a(m, d, \epsilon)$-stability 
immediately implies the $\epsilon$-stability 
for the quotient $q$. 
The map $(C, p_{\bullet}, \widehat{p}_{\bullet}) \mapsto q$
is $S_d$-equivariant, hence we obtain 
the map $\phi$. It is straightforward to check 
that $\phi$ is an isomorphism. 
\end{proof}
\begin{rmk}\label{rmk:coin}
The morphism (\ref{nat:bir})
is $S_d$-equivariant, hence it determines a morphism,
\begin{align*}
c_{\epsilon, \epsilon'} \colon \overline{M}^{\epsilon}_{g, m|d}/S_d
\to \overline{M}^{\epsilon'}_{g, m|d}/S_d. 
\end{align*}
It is easy to see that the above morphism coincides with (\ref{mor:c})
under the isomorphism (\ref{isom:phi}). 
\end{rmk}
\subsection{The case of $(g, m)=(0, 0)$}
Here we investigate $\epsilon$-stable quotients of 
type $(1, 1, d)$ with $(g, m)=(0, 0)$. 
First we take $d$ to be an odd integer with $d=2d'+1$,
$d'\ge 1$. We take $\epsilon_{\bullet}$ as in (\ref{d:odd}). 
Applying the morphism (\ref{nat:bir}) repeatedly, we obtain 
the sequence of birational morphisms, 
\begin{align}\label{seq:bir}
\overline{M}_{0, d}=\overline{M}_{0, d}^{\epsilon_{d'+1}=1}
\to \overline{M}_{0, d}^{\epsilon_{d'}} \to \cdots  
\to \overline{M}_{0, d}^{\epsilon_3} \to \overline{M}_{0, d}^{\epsilon_2 =1/d'}.\end{align}
It is easy to see that $\overline{M}_{0, d}^{1/d'}$ is 
the moduli space of configurations of $d$-points 
in $\mathbb{P}^1$ in which at most $d'$-points 
coincide. This space is 
well-known to be isomorphic to 
the GIT quotient~\cite{Mum}, 
\begin{align}\label{GIT}
\overline{M}_{0, d}^{1/d'} \cong (\mathbb{P}^1)^d 
/\hspace{-.3em}/ \SL_2(\mathbb{C}). 
\end{align}
Here $\SL_{2}(\mathbb{C})$ acts on $(\mathbb{P}^1)^d$ diagonally, 
and we take the linearization on $\oO(\overbrace{1, \cdots, 1}^{d})$
induced by the standard linearization on $\oO_{\mathbb{P}^1}(1)$. 
Since the sequence (\ref{seq:bir}) is $S_d$-equivariant, 
taking the quotients of (\ref{seq:bir}) and combining
the isomorphism (\ref{isom:phi}) yield the sequence of 
birational morphisms, 
\begin{align}\notag
&\overline{Q}_{0, 0}^{\epsilon_{d'+1}=1}
(\pt, d) \to 
\overline{Q}_{0, 0}^{\epsilon_{d'}}
(\pt, d) \to 
\cdots \\
\label{seq:Q}
 & \qquad \qquad \cdots \to 
\overline{Q}_{0, 0}^{\epsilon_{3}}
(\pt, d) \to 
\overline{Q}_{0, 0}^{\epsilon_{2}=1/d'}
(\pt, d) \cong \mathbb{P}^d
/\hspace{-.3em}/ \SL_2(\mathbb{C}).
\end{align}
Here the last isomorphism
is obtained by taking the quotient of (\ref{GIT})
by the $S_d$-action. 
By Remark~\ref{rmk:coin}, 
 each 
morphism in (\ref{seq:Q}) coincides with 
the morphism (\ref{mor:c}). 
Recently 
Kiem-Moon~\cite{KIMO} show that 
 each birational morphism in the
sequence (\ref{seq:bir})
is a blow-up at a union of transversal smooth subvarieties of 
same dimension. As pointed out in~\cite[Remark~4.5]{KM}, 
the sequence (\ref{seq:Q})
is a sequence of weighted blow-ups 
from $\mathbb{P}^d /\hspace{-.3em}/ \SL_2(\mathbb{C})$. 

When $d$ is even
with $d=2d'$, 
let us take $\epsilon_{\bullet}$ as in (\ref{d:even}). 
We also have a similar sequence to (\ref{seq:bir}), 
\begin{align}\notag
\overline{M}_{0, d}=\overline{M}_{0, d}^{\epsilon_{d'}=1}
\to \overline{M}_{0, d}^{\epsilon_{d'-1}} \to \cdots  
\to \overline{M}_{0, d}^{\epsilon_3} \to 
\overline{M}_{0, d}^{\epsilon_2 =1/(d'-1)}.
\end{align}
which is a sequence of blow-ups~\cite{KIMO}.
In this case, instead of the isomorphism (\ref{GIT}), 
there is a birational morphism, (cf.~\cite[Theorem~1.1]{KIMO},)
\begin{align*}
\overline{M}_{0, d}^{1/(d'-1)}
\to (\mathbb{P}^1)^{d}/\hspace{-.3em}/ \SL_2(\mathbb{C}), 
\end{align*} 
obtained by the blow-up along the singular locus which 
consists of $\frac{1}{2}{d \atopwithdelims() d'}$ points in 
the RHS. 
As mentioned in~\cite{KIMO}, 
$\overline{M}_{0, d}^{1/(d'-1)}$ is Kirwan's partial 
desingularization~\cite{FCK}
of the GIT quotient $(\mathbb{P}^1)^{d}/\hspace{-.3em}/ \SL_2(\mathbb{C})$. 
By taking the quotients with respect to the $S_d$-actions, 
we obtain a sequence similar to (\ref{seq:Q}),
\begin{align}\notag
&\overline{Q}_{0, 0}^{\epsilon_{d'}=1}
(\pt, d) \to 
\overline{Q}_{0, 0}^{\epsilon_{d'-1}}
(\pt, d) \to 
\cdots \\
\label{seq:Q2}
 & \qquad \qquad \qquad  \cdots \to 
\overline{Q}_{0, 0}^{\epsilon_{2}=1/(d'-1)}(\pt, d)
\to (\mathbb{P}^d)/\hspace{-.3em}/ \SL_2(\mathbb{C}) ,
\end{align}
a sequence of weighted blow-ups.
Finally Theorem~\ref{thm:cross} yields that 
\begin{align*}
\overline{Q}_{0, 0}^{\epsilon}(\pt, d) =\emptyset,
\quad \epsilon>1 \mbox{ or }d=1. 
\end{align*}
As a summary, we obtain the following. 
\begin{thm}
The moduli space $\overline{Q}_{0, 0}^{\epsilon}(\pt, d)$
is either empty or obtained by a 
sequence of weighted blow-ups starting 
from the GIT quotient 
$\mathbb{P}^{d}/\hspace{-.3em}/ \SL_2(\mathbb{C})$. 
\end{thm}
\subsection{The case of $(g, m)=(0, 1), (0, 2)$}
In this subsection, we study 
moduli spaces of genus zero, $1$ or $2$-pointed 
$\epsilon$-stable quotients 
of type $(1, 1, d)$.
Note that for small $\epsilon$, we have 
\begin{align*}
\overline{Q}_{0, 1}^{\epsilon}(\pt, d)=\emptyset, \quad 
0<\epsilon \le 1/d. 
\end{align*}
The first interesting situation 
happens at $\epsilon=1/(d-1)$
and $d\ge 2$. 
For an object
\begin{align*}
(C, p, \widehat{p}_1, \cdots, \widehat{p}_d) \in 
\overline{M}^{1/(d-1)}_{0, 1|d}, 
\end{align*}
applying 
Lemma~\ref{lem:ob} 
immediately implies that $C\cong \mathbb{P}^1$. 
We may assume that $p=\infty \in \mathbb{P}^1$, 
hence $\widehat{p}_i \in \mathbb{A}^1$. 
The stability condition is equivalent to 
that at least two points 
among $\widehat{p}_1, \cdots, \widehat{p}_d$ 
are distinct.
Let $\Delta$ be the small diagonal, 
\begin{align*}
\Delta=\{(x_1, \cdots, x_d)\in \mathbb{A}^d :
x_1=x_2=\cdots =x_d\}.
\end{align*}
 Noting that the subgroup of 
automorphisms of $\mathbb{P}^1$ preserving 
$p\in \mathbb{P}^1$ is $\mathbb{A}^1 \rtimes \mathbb{G}_m$, 
we have 
\begin{align*}
\overline{M}^{1/(d-1)}_{0, 1|d}
&\cong (\mathbb{A}^{d}\setminus \Delta)/\mathbb{A}\rtimes 
\mathbb{G}_m \\
&\cong \mathbb{P}^{d-2}. 
\end{align*}
By Proposition~\ref{prop:MOP}, we obtain 
\begin{align}\label{gm01}
\overline{Q}_{0, 1}^{1/(d-1)}(\pt, d)
\cong \mathbb{P}^{d-2}/S_d. 
\end{align}
In particular 
for each $\epsilon \in \mathbb{R}_{>0}$, the 
moduli space 
 $\overline{Q}_{0, 1}^{\epsilon}(\pt, d)$
is either empty or admits a birational 
morphism to $\mathbb{P}^{d-2}/S_d$. 

Next we look at the case of $(g, m)=(0, 2)$. 
An $\epsilon$-stable quotient is a MOP-stable 
quotient for $0<\epsilon \le 1/d$, and 
in this case the 
moduli space is 
described in~\cite[Section~4]{MOP}. 
In fact for any MOP-stable quotient 
$\oO_C^{\oplus n} \stackrel{q}{\twoheadrightarrow}Q$, 
the curve $C$ is a chain of rational curves and 
two marked points lie at distinct rational tails
if $C$ is not irreducible. 
If $k$ is the number of irreducible components of 
$C$, then giving a MOP stable quotient is equivalent 
to giving a partition $d_1+\cdots +d_k=d$ and 
length $d_i$-divisors on each 
irreducible component up to rotations. Therefore we have 
(set theoretically)
\begin{align}\label{set}
\overline{Q}_{0, 2}(\pt, d)
=\coprod_{\begin{subarray}{c}k\ge 1 \\
d_1+\cdots +d_k=d \end{subarray}}
\prod_{j=1}^{k}\Sym ^{d_i}(\mathbb{C}^{\ast})/\mathbb{C}^{\ast}.
\end{align}
For $1/d<\epsilon \le 1/(d-1)$, 
a MOP stable quotient $\oO_C^{\oplus n} \stackrel{q}{\twoheadrightarrow}Q$
is not $\epsilon$-stable if and only if 
$C\cong \mathbb{P}^1$ and 
the support of $\tau(Q)$ consists of one point. 
Such stable quotients consist of one point in 
the RHS of (\ref{set}). 
Noting the isomorphism (\ref{gm01}), 
the Cartesian diagram (\ref{dig:Car}) is described 
as follows, 
\begin{align*}
  \xymatrix{
  \mathbb{P}^{d-2}/S_d
 \ar[r]\ar[d] &
  \overline{Q}_{0, 2}^{\epsilon}(\pt, d) 
  \ar[d], \\
 \Spec \mathbb{C}
 \ar[r] &
  \overline{Q}_{0, 2}(\pt, d). }
  \end{align*}

\section{Proof of Theorem~\ref{thm:rep}}\label{sec:proof}
In this section, we give a proof of Theorem~\ref{thm:rep}. 
We first show that $\overline{Q}_{g, m}^{\epsilon}(\mathbb{G}(r, n), d)$
is a Deligne-Mumford stack of finite type over 
$\mathbb{C}$, following the 
argument of~\cite{MOP}, \cite{BH}. 
Next we show the properness of 
$\overline{Q}_{g, m}^{\epsilon}(\mathbb{G}(r, n), d)$
using the valuative criterion. The argument
to show the properness of MOP-stable quotients~\cite[Section~6]{MOP}
is not 
 applied for $\epsilon$-stable quotients. 
Instead we give an alternative argument, 
which also gives another proof of~\cite[Theorem~1]{MOP}. 
\subsection{Construction of the moduli space}
The same arguments of Proposition~\ref{prop:wall} and Proposition~\ref{wall2}
show the similar result 
for the 2-functors (\ref{2func}). 
For $\epsilon >1$, the moduli space 
of $\epsilon$-stable quotients is 
either empty or isomorphic to the moduli space 
of stable maps to the Grassmannian. 
Therefore we assume that 
\begin{align*}
\epsilon=\frac{1}{l}, \quad l=1, 2, \cdots, d,
\end{align*}
and construct the moduli space $\overline{Q}_{g, m}^{1/l}(\mathbb{G}(r, n), d)$
as a global quotient stack. If $\epsilon=1/d$, then 
the moduli space coincides with that of MOP-stable 
quotients,
(cf.~Theorem~\ref{thm:cross},)
 and the construction is given in~\cite[Section~6]{MOP}.
We need to slightly modify the argument 
to construct the moduli spaces for a general $\epsilon$, 
but the essential idea is the same. 
First we show the following lemma. 
\begin{lem}\label{veryample}
Take an $\epsilon=1/l$-stable quotient 
$\oO_C^{\oplus n}\stackrel{q}{\twoheadrightarrow}Q$
and an integer $k\ge 5$. Then the line bundle 
$\lL(q, 1/l)^{\otimes lk}$ 
is very ample. Here $\lL(q, 1/l)$
is defined in (\ref{def:L}).
\end{lem}
\begin{proof}
It is enough to show that for $x_1, x_2 \in C$, we have 
\begin{align}\label{x1x2}
H^1(C, \lL(q, 1/l)^{\otimes lk}\otimes I_{x_1}I_{x_2})=0. 
\end{align}
Here $I_{x_i}$ is the ideal sheaf of $x_i$. 
By the Serre duality, (\ref{x1x2}) is equivalent to 
\begin{align}\label{xx12}
\Hom(I_{x_1}I_{x_2}, \omega_{C}\otimes \lL(q, 1/l)^{\otimes (-lk)})=0. 
\end{align}
Suppose that $x_1, x_2 \in C^{ns}$. 
For an irreducible component $P\subset C$, 
we set $d_P=\deg(Q|_{P})$. 
In the notation of Lemma~\ref{lem:ob}, 
we have 
\begin{align}\notag
&\deg(\omega_C(x_1+x_2)\otimes \lL(q, 1/l)^{\otimes (-lk)}|_{P}) \\
\notag
&\le 2g(P)-2+s(P)+2-lk(2g(P)-2+s(P)+d_{P}/l) \\
\label{deg:ineq}
&=(2g(P)-2+s(P))(1-lk)+2-d_{P}k. 
\end{align}
In the case of 
\begin{align*}
2g(P)-2+s(P)>0, 
\end{align*}
then (\ref{deg:ineq}) is 
obviously negative. Otherwise 
$(g(P), s(P))$ is either one of the following, 
\begin{align*}
(g(P), s(P))=(1, 0), (0, 2), (0, 1), (0, 0).
\end{align*} 
In these cases, (\ref{deg:ineq}) is negative 
by Lemma~\ref{lem:ob}.
Therefore (\ref{xx12}) holds.  

When $x_1$ or $x_2$ or both of them are node, 
for instance $x_1$ is node and $x_2 \in C^{ns}$, 
then we take the normalization at $x_1$, 
\begin{align*}
\pi \colon \widetilde{C}\to C, 
\end{align*}
with $\pi^{-1}(x_1)=\{x_1', x_1''\}$. 
Then (\ref{xx12}) is equivalent to 
\begin{align}\label{xxx12}
H^0(\widetilde{C}, \omega_{\widetilde{C}}(x_1'+x_1''+x_2)\otimes 
\lL(q, 1/l)^{\otimes (-lk)})=0,
\end{align}
and the same calculation as above shows (\ref{xxx12}). 
The other cases are also similarly discussed. 
\end{proof}
By Lemma~\ref{veryample}, 
we have
\begin{align}\label{notdepend}
h^{0}(C, \lL(q, 1/l)^{\otimes kl})
=1-g+kl(2g-2)+kd+m,
\end{align}
which does not depend on a choice of $1/l$-stable 
quotient of type $(r, n, d)$. 
Let $V$ be an $\mathbb{C}$-vector space of 
dimension (\ref{notdepend}). 
The very ample line bundle $\lL(q, 1/l)^{\otimes kl}$
on $C$
determines an embedding, 
\begin{align*}
C\hookrightarrow \mathbb{P}(V),
\end{align*}
and marked points determine points in $\mathbb{P}(V)$. 
Therefore $1/l$-stable quotient associates
a point, 
\begin{align}
(C, p_1, \cdots, p_m) \in \Hilb(\mathbb{P}(V)) \times 
\mathbb{P}(V)^{\times m}. 
\end{align}
Let 
\begin{align*}
\hH \subset \Hilb(\mathbb{P}(V))\times \mathbb{P}(V)^{\times m}
\end{align*}
be the locally closed subscheme which parameterizes 
$(C, p_1, \cdots, p_m)$ satisfying the following. 
\begin{itemize}
\item The subscheme $C\subset \mathbb{P}(V)$ is a connected nodal 
curve of genus $g$. 
\item We have $p_i \in C^{ns}$ and $p_i \neq p_j$
for $i\neq j$. 
\end{itemize}
Let $\pi \colon \cC \to \hH$ be the universal curve and 
\begin{align*}
\Quot(n-r, d) \to \hH 
\end{align*}
the relative Quot scheme which parameterizes 
rank $n-r$, degree $d$ quotients 
$\oO_C^{\oplus n} \twoheadrightarrow Q$
on the fibers of $\pi$. 
We define 
\begin{align*}
\qQ \subset \Quot(n-r, d),
\end{align*}
to be the locally closed subscheme
corresponding to quotients 
$\oO_C^{\oplus n} \stackrel{q}{\twoheadrightarrow} Q$
satisfying the following. 
\begin{itemize}
\item The coherent sheaf $Q$ is locally free 
near nodes and $p_i$. 
\item For any $p\in C$, we have $\length \tau(Q)_{p} \le l$. 
\item The line bundle $\lL(q, 1/l)^{\otimes lk}$ 
coincides with $\oO_{\mathbb{P}(V)}(1)|_{C}$. 
\end{itemize}
The natural $\PGL_{2}(\mathbb{C})$-action on 
$\hH$ lifts to the action on $\qQ$, and 
the desired moduli space is 
the following
quotient stack, 
\begin{align*}
\overline{Q}_{g, m}^{1/l}(\mathbb{G}(r, n), d)
=[\qQ/\PGL_{2}(\mathbb{C})].
\end{align*}
By Lemma~\ref{lem:aut}, the stabilizer groups of 
closed points in 
$\overline{Q}_{g, m}^{1/l}(\mathbb{G}(r, n), d)$
are finite. Hence this is a Deligne-Mumford stack
of finite type over $\mathbb{C}$. 
\subsection{Valuative criterion}
In this subsection, we prove the 
properness of the moduli stack 
$\overline{Q}_{g, m}^{\epsilon}(\mathbb{G}(r, n), d)$. 
Before this, we introduce some notation. 
Let $X$ be a variety and $F$ a locally 
sheaf of rank $r$ on $X$. 
For $n\ge r$ and a morphism,
\begin{align*}
s\colon \oO_X^{\oplus n} \to F, 
\end{align*}
we associate the degenerate locus,
\begin{align*}
Z(s)\subset X.
\end{align*}
Namely $Z(s)$ is defined by
the ideal, locally generated by $r\times r$-minors of 
the matrix given by $s$. 
For a point $g\in \mathbb{G}(r, n)$, 
let us choose a lift of $g$ to an embedding 
\begin{align}\label{lift}
g\colon \mathbb{C}^{r} \hookrightarrow \mathbb{C}^n.
\end{align}
Here by abuse of notation, we have also denoted 
the above embedding by $g$. 
We have the
sequence, 
\begin{align*}s_g \colon 
\oO_{X}^{\oplus r}\stackrel{g}{\hookrightarrow}
\oO_X^{\oplus n} \stackrel{s}{\to} F.
\end{align*}
The morphism $s_g$ is determined 
by $g\in \mathbb{G}(r, n)$
up to the $\GL_{r}(\mathbb{C})$-action on 
$\oO_X^{\oplus r}$. 
Note that if $s_g$ is injective, then 
$Z(s_g)$ is a 
divisor on $X$ 
which does not depend on a choice of 
a lift (\ref{lift}).  
The divisor $Z(s_g)$ fits into 
the exact sequence, 
\begin{align*}
0 \to \bigwedge^{r}F^{\vee} \to \oO_X 
\to \oO_{Z(s_g)} \to 0. 
\end{align*}
When $X=\mathbb{G}(n-r, n)$ and 
$s$ is a universal rank $r$ quotient, 
then $s_g$ is injective and $H_g \cneq Z(s_g)$ 
is a divisor in $\mathbb{G}(n-r, n)$. 
\begin{lem}\label{lem:div}
Let $\oO_C^{\oplus n}\stackrel{q}{\twoheadrightarrow} Q$
be an $\epsilon$-stable quotient with kernel $S$ and 
marked points $p_1, \cdots, p_m$. 
Let $s\colon \oO_{C}^{\oplus n}\to S^{\vee}$ be the dual 
of the inclusion $S\hookrightarrow \oO_C^{\oplus n}$. 
Then for a general choice of $g\in \mathbb{G}(r, n)$, 
the degenerate locus
$Z(s_g)\subset C$ is a divisor written as
\begin{align*}
Z(s_g)=Z(s)+D_g. 
\end{align*}
Here $D_g$ is a reduced divisor on $C$ satisfying 
\begin{align*}
D_g \cap \{Z(s)\cup \{p_1, \cdots, p_m\} \}=\emptyset. 
\end{align*}
\end{lem}
\begin{proof}
Let $F\subset S^{\vee}$ be the image of $s$. 
Note that $F$ is a locally free sheaf of rank $r$, hence 
it determines a map, 
\begin{align*}
\pi_{F}\colon C \to \mathbb{G}(n-r, n). 
\end{align*}
It is easy to see that a general $g\in \mathbb{G}(r, n)$
satisfies the following. 
\begin{itemize}
\item The divisor $H_g \subset \mathbb{G}(n-r, n)$
intersects the image of $\pi_{F}$ transversally. 
(Or the intersection is empty if $\pi_{F}(C)$ is a point.)
\item For $p\in \Supp \tau(Q) \cup \{p_1, \cdots, p_m\}$, 
we have $\pi_{F}(p)\notin H_g$. 
\end{itemize}
Then we have 
\begin{align*}
Z(s_g)=Z(s)+\pi_{F}^{\ast}H_g,
\end{align*}
and $D_g \cneq \pi_{F}^{\ast}H_g$ satisfies 
the desired property. 
\end{proof}
In the next proposition, we show that 
the moduli space of $\epsilon$-stable quotients
is separated. 
Let $\Delta$ be a non-singular
curve with a closed point $0\in \Delta$. We set 
\begin{align*}
\Delta^{\ast}=\Delta \setminus\{0\}.
\end{align*}
\begin{prop}\label{qstable}
For $i=1, 2$, let $\pi_i \colon \xX_i \to \Delta$ be 
flat families of quasi-stable curves with 
disjoint sections 
$p_1^{(i)}, \cdots, p_m^{(i)} \colon \Delta \to \xX_i$. 
Let $q_i \colon \oO_{\xX_i}^{\oplus n} \twoheadrightarrow \qQ_i$
be flat families of $\epsilon$-stable quotients
of type $(r, n, d)$ 
which are isomorphic over $\Delta^{\ast}$. 
Then possibly after base change ramified over $0$, 
there is an isomorphism 
$\phi \colon \xX_1 \stackrel{\sim}{\to}\xX_2$
over $\Delta$ and an isomorphism 
$\psi \colon \phi^{\ast}\qQ_2 \stackrel{\sim}{\to}
\qQ_1$
such that the following diagram commutes, 
\begin{align*}
\xymatrix{
\oO_{\xX_1}^{\oplus n} \ar[r]^{\phi^{\ast}q_2} \ar[d]_{\id}
& \phi^{\ast}\qQ_2 \ar[d]^{\psi} \\
\oO_{\xX_1}^{\oplus n} \ar[r]^{q_1} & \qQ_1.}
\end{align*}
\end{prop}
\begin{proof}
Since the relative Quot scheme is separated, it 
is enough show that the isomorphism over $\Delta^{\ast}$ extends to the 
families of marked curves $\pi_i \colon \xX_i \to \Delta$.
By taking the base change and the normalization, 
we may assume that the general fibers of $\pi_i$ are 
non-singular irreducible curves, by adding the preimage
of the nodes to the marking points. 
Let us take exact sequences, 
\begin{align*}
0 \to \sS_i \to \oO_{\xX_i}^{\oplus n} \to \qQ_i \to 0. 
\end{align*}
Since $\sS_i|_{\xX_{i, t}}$ is locally free 
for any $t\in \Delta$, 
where $\xX_{i, t}\cneq \pi_{i}^{-1}(t)$, 
the sheaf $\sS_i$ is
a locally free sheaf on $\xX_i$. 
Taking the dual, we obtain the morphism, 
\begin{align*}
s_i \colon \oO_{\xX_i}^{\oplus n}\to 
\sS_i^{\vee}.
\end{align*}
Let us take a general point $g\in \mathbb{G}(r, n)$
and the degenerate locus,
\begin{align*}
D_i \cneq Z(s_{i, g})\subset
\xX_i.
\end{align*}
By Lemma~\ref{lem:div}, the divisor 
$D_{i, t}\cneq D_i|_{\xX_{i, t}}$ is written as 
\begin{align*}
D_{i, t}=Z(s_{i, t})+D_{i, t}^{\circ},
\end{align*}
where $D_{i, t}^{\circ}$ is a reduced 
divisor on $\xX_{i, t}$, satisfying 
\begin{align*}
D_{i, t}^{\circ} \cap \{ Z(s_{i, t}) \cup \{p_1(t), \cdots, p_m(t)\} \}
=\emptyset.
\end{align*}
Then the $\epsilon$-stability 
of $\oO_{\xX_{i, t}}^{\oplus n} \stackrel{q_{i, t}}{\twoheadrightarrow}
\qQ_{i}|_{\xX_{i, t}}$
implies the following. 
\begin{itemize}
\item The coefficients of the
$\mathbb{R}$-divisor $\sum_{j=1}^{m}p_j^{(i)}(t) +\epsilon \cdot D_{i, t}$
have less than or equal to $1$. 
\item The $\mathbb{R}$-divisor 
$K_{\xX_{i, t}}+\sum_{j=1}^{m}
p_j^{(i)}(t) +\epsilon \cdot D_{i, t}$ is ample 
on $\xX_{i, t}$. 
\end{itemize}
The first condition implies that the pairs
\begin{align}\label{quot:pair}
(\xX_i, \sum_{j=1}^{m}p_j^{(i)}+\epsilon \cdot D_i), \quad i=1, 2,
\end{align}
have only log canonical singularities.
(cf.~\cite{KM}, \cite{KMM}.)
Also since the divisors 
$\sum_{j=1}^{m}p_j^{(i)}+\epsilon \cdot D_i$
do not contain curves supported on the 
central fibers, 
we have 
\begin{align*}
\phi_{\ast}\left(\sum_{j=1}^{m}p_j^{(1)}+\epsilon \cdot D_1\right)
=\sum_{j=1}^{m}p_j^{(2)}+\epsilon \cdot D_2,
\end{align*}
where $\phi$ is the birational map 
$\phi \colon \xX_1 \dashrightarrow \xX_2$. 
Therefore the pairs (\ref{quot:pair}) are birational 
log canonical models over $\Delta$. 
Since two birational log canonical models
are isomorphic,
the birational map $\phi$
extends to an isomorphism $\phi \colon \xX_1 \stackrel{\cong}{\to} \xX_2$.  
\end{proof}
Finally we show that the moduli space 
$\overline{Q}_{g, m}^{\epsilon}(\mathbb{G}(r, n), d)$
is complete. 
\begin{prop}
Suppose that the following data is a flat family of $m$-pointed 
$\epsilon$-stable quotients
of type $(r, n, d)$ over $\Delta^{\ast}$, 
\begin{align}\label{given}
\pi^{\ast} \colon \xX^{\ast} \to \Delta^{\ast}, \quad p_1^{\ast},
 \cdots, p_m^{\ast} \colon 
\Delta^{\ast} \to \xX^{\ast}, \quad q^{\ast}\colon \oO_{\xX^{\ast}}^{\oplus n}
\twoheadrightarrow \qQ^{\ast}.
\end{align}
Then possibly after base change
ramified over $0\in \Delta$, there is 
a flat family of $m$-pointed $\epsilon$-stable quotients over 
$\Delta$, 
\begin{align}\label{ext-fami}
\pi \colon \xX \to \Delta, \quad p_1, \cdots, p_m \colon 
\Delta \to \xX, \quad q\colon \oO_{\xX}^{\oplus n} \twoheadrightarrow \qQ,
\end{align}
which is isomorphic to (\ref{given})
over $\Delta^{\ast}$. 
\end{prop}
\begin{proof}
As in the proof of Proposition~\ref{qstable}, 
we may assume that 
 the general fibers of $\pi^{\ast}$ are 
non-singular irreducible curves.
Let $\sS^{\ast}$ be the kernel of $q^{\ast}$. 
Taking the dual of the inclusion
$\sS^{\ast} \subset \oO_{\xX^{\ast}}^{\oplus n}$, 
we obtain the morphism 
\begin{align*}
s^{\ast}\colon \oO_{\xX^{\ast}}^{\oplus n}
\to \sS^{\ast \vee}.
\end{align*}
We choose a general point,
\begin{align}\label{choose}
g\in \mathbb{G}(r, n),
\end{align}
and set 
$D^{\ast}\cneq Z(s^{\ast}_{g}) \subset \xX^{\ast}$. 
As in the proof of Proposition~\ref{qstable},
the $\epsilon$-stability implies that the pair 
\begin{align}\label{indeed}
(\xX^{\ast}, \sum_{j=1}^{m}p_{j}^{\ast}+\epsilon \cdot 
D^{\ast})
\end{align}
is a log canonical model over $\Delta^{\ast}$. 

Indeed, the family (\ref{indeed}) 
can be interpreted as a family of Hassett's weighted 
pointed stable curves~\cite{BH}. 
Let us write 
\begin{align*}
D^{\ast}=\sum_{j=1}^{k}m_j D_j^{\ast},
\end{align*}
for distinct irreducible divisors $D_j^{\ast}$ and $m_j \ge 1$. 
Since the family (\ref{given}) is of type $(r, n, d)$, we have 
\begin{align*}
m_1+ m_2+ \cdots +m_k=d. 
\end{align*}
By shrinking $\Delta$ if necessary, we may assume that 
each $D_j^{\ast}$ is a section of $\pi^{\ast}$. 
Then the data
\begin{align}\label{fam:wei}
(\pi^{\ast} \colon \xX^{\ast}\to \Delta^{\ast}, 
p_1^{\ast}, \cdots, p_m^{\ast}, 
\overbrace{D_1^{\ast}, \cdots, D_1^{\ast}}^{m_1}, \cdots, 
\overbrace{D_k^{\ast}, \cdots, D_k^{\ast}}^{m_k}), 
\end{align}
is a family of 
$a(m, d, \epsilon)$-stable $m+d$-pointed curves~\cite{BH}
over $\Delta^{\ast}$. 
(See Definition~\ref{def:weighted} and (\ref{amd}).)
By the properness of $\overline{M}_{g, m|d}^{\epsilon}$, 
(cf.~\cite{BH}, (\ref{denoted}),)
there is a family of $a(m, d, \epsilon)$-stable
$m+d$-pointed curves over $\Delta$, 
\begin{align}\label{a(mde)}
(\pi \colon \xX \to \Delta, 
p_1, \cdots, p_m, 
\overbrace{D_1, \cdots, D_1}^{m_1}, \cdots, 
\overbrace{D_k, \cdots, D_k}^{m_k}), 
\end{align}
which is isomorphic to the family (\ref{fam:wei}) over $\Delta^{\ast}$. 
In particular we have an extension of $D^{\ast}$ to
$\xX$, 
\begin{align*}
D=\sum_{j=1}^{k}m_j D_j,  \quad 
D|_{\xX^{\ast}}=D^{\ast}.  
\end{align*}

By the properness of the relative Quot scheme, there is an 
exact sequence 
\begin{align}\label{Quot-ex}
0 \to \sS \to \oO_{\xX}^{\oplus n}\stackrel{q}{\to} \qQ \to 0, 
\end{align}
such that $q$ is isomorphic to $q^{\ast}$ over $\Delta^{\ast}$. 
Restricting to $\xX_0$, we obtain the exact sequence, 
\begin{align}\label{seq:res}
0 \to \sS_0 \to \oO_{\xX_0}^{\oplus n} \stackrel{q_0}{\to} \qQ_0 \to 0. 
\end{align}
We claim that the quotient
$q_0$ is 
an $\epsilon$-stable quotient, hence 
the family $(\xX, p_1, \cdots, p_m)$ and $q$ 
gives a desired extension (\ref{ext-fami}). 
We prove the following lemma. 
\begin{lem}
The sheaf $\sS$ is a locally free sheaf on $\xX$. 
\end{lem}
\begin{proof}
First we see that 
the sheaf $\sS$ is reflexive, i.e. $\sS^{\vee \vee}\cong \sS$. 
We have the morphism of exact sequence of sheaves on $\xX$, 
\begin{align*}
\xymatrix{0 \ar[r] &
\sS \ar[r]\ar[d] & \oO_{\xX}^{\oplus n} \ar[r]\ar[d] & 
\qQ \ar[r]\ar[d] & 0 \\
0 \ar[r] &
\sS^{\vee \vee} \ar[r] & \oO_{\xX}^{\oplus n} \ar[r] & \qQ' \ar[r]
& 0,
}
\end{align*}
where the left arrow is an injection. 
By the snake lemma, there is an inclusion, 
\begin{align*}
\sS^{\vee \vee}/\sS \hookrightarrow \qQ, 
\end{align*}
and $\sS^{\vee \vee}/\sS$ is supported on $\xX_0$, 
which contradicts to that $\qQ$ is flat over $\Delta$. 
In particular setting 
\begin{align*}
U=\xX \setminus (\mbox{nodes of }\xX_0),
\end{align*}
the sheaf $\sS$ is a push-forward of some locally free 
sheaf on $U$ to $\xX$. 
We only need to check that $\sS$ is free at nodes on 
$\xX_0$. 

Taking the dual of the inclusion
$\sS \hookrightarrow \oO_{\xX}^{\oplus n}$ and 
composing with $g\colon \oO_{\xX}^{\oplus r} \hookrightarrow 
\oO_{\xX}^{\oplus n}$, 
where $g$ is taken in (\ref{choose}), 
we obtain a morphism 
\begin{align*}
s_g \colon \oO_{\xX}^{\oplus r} \stackrel{g}{\hookrightarrow}
\oO_{\xX}^{\oplus n} \to \sS^{\vee}. 
\end{align*}
Restricting to $U$, we obtain the divisor in $U$,
\begin{align*}
D^{\dag}_{U}\cneq Z(s_g|_{U}) \subset U,
\end{align*}
and the closure of $D^{\dag}_{U}$ in $\xX$ is 
denoted by $D^{\dag}$. We have the following. 
\begin{itemize}
\item By the construction, we have 
$D|_{\xX^{\ast}}=D^{\dag}|_{\xX^{\ast}}$.
\item Replacing $g$ by another general point 
in $\mathbb{G}(r, n)$ if necessary, the divisors $D^{\dag}$
and $D$
do not contain any irreducible component of $\xX_0$. 
\end{itemize}
These properties imply that $D^{\dag}=D$. 
Noting that the divisor $D$ has support away from 
nodes of $\xX_0$, the support of the cokernel 
of $s_g$ is written as 
\begin{align}\label{supp:cok}
\Supp \Cok(s_g)=\Supp(D) \amalg V, 
\end{align}
where $V$ is a finite set of points contained in 
the nodes of $\xX_0$. However if $V$ is non-empty, then 
there is a nodal point $x\in \xX_{0}$ and an injection 
$\oO_x \hookrightarrow \sS^{\vee}$, which 
contradicts to that $\sS$ is torsion free. 
Therefore $V$ is empty, and the morphism $s_g$ is 
isomorphic on nodes of $\xX_0$. 
Hence $\sS^{\vee}$ is a locally free sheaf on $\xX_{0}$,
 and the sheaf $\sS$
is also locally free
since $\sS \cong \sS^{\vee \vee}$. 
\end{proof}
Note that the locally freeness of $\sS$ implies 
that the divisor $Z(s_g)$
is well-defined, and the 
proof of the above lemma 
immediately implies that 
\begin{align}\label{Z=D}
Z(s_g)=D^{\dag}=D. 
\end{align}

Next let us see that $q_0$ is a quasi-stable quotient. 
Taking $\hH om(\ast, \oO_{\xX_0})$ to 
the exact sequence (\ref{seq:res}), we obtain 
the exact sequence, 
\begin{align*}
0 \to \qQ_0^{\vee} \to \oO_{\xX_0}^{\oplus n} \stackrel{s_0}{\to} \sS_0^{\vee} 
\to \eE xt^1_{\xX_0}(\qQ_0, \oO_{\xX_0}) \to 0,
\end{align*}
and the vanishing $\eE xt^{i}_{\xX_0}(\qQ_0, \oO_{\xX_0})=0$
for $i\ge 2$. 
We have the surjection, 
\begin{align}\label{RHSsur}
\Cok(s_{0, g}) \twoheadrightarrow \eE xt^{1}_{\xX_0}(\qQ_0, \oO_{\xX_0}),
\end{align}
and the LHS of (\ref{RHSsur}) has support away from 
nodes and markings by (\ref{Z=D}). 
Therefore for a nodal point or marked point 
$p\in \xX_{0}$, we have 
\begin{align*}
\eE xt^{i}_{\xX_0}(\qQ_0, \oO_{\xX_0})_{p}=0, \quad 
i\ge 1, 
\end{align*}
which implies that $\qQ_0$ is locally free at 
$p$, i.e. 
$q_0 \colon \oO_{\xX_0}^{\oplus n} \twoheadrightarrow \qQ_0$
is a quasi-stable quotient. 

Finally we check the $\epsilon$-stability of $q_0$. 
The ampleness of $\lL(q_{0}, \epsilon)$ is equivalent 
to the ampleness of the divisor, 
\begin{align}\label{eq:div}
K_{\xX_0}+p_1(0)+\cdots +p_m(0)+\epsilon \cdot Z(s_{g, 0}). 
\end{align}
Noting the equality (\ref{Z=D}), we have 
$Z(s_{g, 0})=D|_{\xX_0}$. 
Since the data (\ref{a(mde)}) is a family 
of $a(m, d, \epsilon)$-stable curves, 
the divisor (\ref{eq:div}) on $\xX_0$ is ample. 
Also the surjection (\ref{RHSsur}) 
and the fact $Z(s_{g, 0})=D|_{\xX_0}$ imply that 
\begin{align}\notag
\epsilon \cdot \length \tau(\qQ_0)_{p} &\le \epsilon \cdot
\length \Cok(s_{g, 0})_{p} \\
\notag
&= \epsilon \cdot \length \oO_{Z_{s_{g, 0}}, p}, \\
\label{length}
&= \epsilon \cdot \length \oO_{D}|_{\xX_0, p},
\end{align}
for any $p\in \xX_0$. Again 
noting that
(\ref{a(mde)}) is a family of $a(m, d, \epsilon)$-stable 
curves, we 
conclude that $(\ref{length}) \le 1$. 
Therefore $q_0$ is an $\epsilon$-stable quotient. 
\end{proof}

\section{Wall-crossing formula}\label{sec:cycle}
The purpose of this section is to give an 
argument to prove Theorem~\ref{thm:wcf}. 
Our strategy is to modify ~\cite[Secton~7]{MOP}
so that $\epsilon$
is involved in the argument. 
Therefore we only focus 
on the arguments to be modified, 
and we leave several details to the reader. 
\subsection{Localization}\label{subsec:locl}
Let $T$ be a torus $T=\mathbb{G}_m^{n}$ 
acting on $\mathbb{C}^n$ via 
\begin{align*}
(t_1, \cdots, t_n) \cdot (x_1, \cdots, x_n)
=(t_1x_1, \cdots, t_n x_n). 
\end{align*}
The above $T$-action induces a $T$-action on 
$\mathbb{G}(r, n)$ and $\overline{Q}_{g, m}^{\epsilon}(\mathbb{G}(r, n), d)$. 
Over the moduli space of
MOP-stable quotients, the $T$-fixed 
loci are obtained in~\cite[Section~7]{MOP}
via certain combinatorial data. 
The $T$-fixed loci of $\epsilon$-stable quotients are 
similarly obtained, but we need to take 
the $\epsilon$-stability into consideration. 
 They are indexed by the following data, 
\begin{align}\label{ind:graph}
\theta=(\Gamma, \iota, \gamma, s, \beta, \delta, \mu).
\end{align}
\begin{itemize}
\item $\Gamma=(V, E)$ is a connected graph, 
 where $V$ is the vertex set and $E$ is 
the edge set with no self edges.  
\item $\iota$ is an 
assignment of an inclusion,
\begin{align*}
\iota_{v}\colon \{1, \cdots, r\} \to \{1, \cdots, n\}, 
\end{align*}
to each $v\in V$. In particular, the induced subspace
$\mathbb{C}^{r} \hookrightarrow \mathbb{C}^n$
by $\iota_{v}$ determines a map, 
\begin{align*}
\nu \colon V \to \mathbb{G}(r, n)^{T}.
\end{align*}
\item $\gamma$ is a genus assignment $\gamma \colon V\to \mathbb{Z}_{\ge 1}$, satisfying 
\begin{align*}
\sum_{v\in V}\gamma(v)+h^{1}(\Gamma)=g. 
\end{align*}
\item For each $v\in V$, 
$s(v)=(s_1(v), \cdots, s_r(v))$ with $s_i(v)\in \mathbb{Z}_{\ge 0}$. 
We set 
\begin{align*}
{\bf s}(v)=\sum_{i=1}^{r}s_i(v).
\end{align*}
\item $\beta$ is an assignment to each $e\in E$ of a $T$-invariant 
curve $\beta(e)$ of $\mathbb{G}(r, n)$.
The two vertices incident to 
$e\in E$ are mapped via $\nu$ to the two $T$-fixed points incident 
to $\beta(e)$. 
\item $\delta \colon E \to \mathbb{Z}_{\ge 1}$ is an assignment 
of a covering number, satisfying 
\begin{align*}
\sum_{v\in V}{\bf s}(v)+\sum_{e\in E}\delta(e)=d. 
\end{align*} 
\item $\mu$ is a distribution of the $m$-markings to the vertices of $V$. 
\item For each $v\in V$, we set 
\begin{align*}
w(v)=\min\{0, 2\gamma(v)-2+\epsilon \cdot {\bf s}(v)+\valence(v)  \}. 
\end{align*}
Then for each edge $e\in E$ with incident 
vertex $v_1, v_2 \in V$, we have 
\begin{align}\label{graph:ample}
\epsilon \cdot \delta(e)+w(v_1)+w(v_2)>0. 
\end{align}
\end{itemize}
The condition (\ref{graph:ample}) 
corresponds to the
ampleness of (\ref{def:L}) 
 at the irreducible component determined by $e$. 
Given a data $\theta$ as in (\ref{ind:graph}),
the isomorphism classes of $T$-fixed 
$\epsilon$-stable quotients indexed by $\theta$ form
a product of the quotients of the moduli spaces of 
weighted pointed stable curves, 
\begin{align}\label{wei:poi}
Q^{T}(\theta)=
\prod_{v\in V}\left(\overline{M}_{\gamma(v), \valence(v)|{\bf s}(v)}^{\epsilon}/\Pi_{i=1}^{r}S_{s_i(v)}\right). 
\end{align}
Here if $v\in V$ does not satisfy the condition, 
\begin{align}\label{pos}
2\gamma(v)-2+\epsilon \cdot s(v)+\valence(v) >0, 
\end{align}
we set 
\begin{align*}
\overline{M}_{\gamma(v), \valence(v)|{\bf s}(v)}^{\epsilon}=
\left\{ 
\begin{array}{cc}
\Spec \mathbb{C}, & V \neq \{ v \}, \\
\emptyset, & V=\{v\}.
\end{array} \right.
\end{align*}
The corresponding $T$-fixed $\epsilon$-stable quotients 
are described in the following way. 
\begin{itemize}
\item For $v\in V$, suppose that the condition 
(\ref{pos}) holds. 
A point in the $v$-factor of (\ref{wei:poi})
determines a curve $C_v$ and $r$-tuple of divisors on it
$D_1, \cdots, D_r$ with $\deg(D_i)=s_i(v)$. 
Then an $\epsilon$-stable quotient is obtained by 
the exact sequence,
\begin{align}\label{tuple}
0 \to \oplus_{i=1}^{r}\oO_{C_v}(-D_i) \to \oO_{C_v}^{\oplus n} \to Q \to 0. 
\end{align}
Here the first inclusion is the composition of the 
natural inclusion 
\begin{align*}
\oplus_{i=1}^{r}\oO_{C_v}(-D_i) \hookrightarrow \oO_{C_v}^{\oplus r},
\end{align*}
and the inclusion $\oO_{C_v}^{\oplus r} \hookrightarrow \oO_{C_v}^{\oplus n}$
induced by $\iota_v$. 
\item For $e\in E$, 
 consider the degree $\delta(e)$-covering
ramified over the two torus fixed points, 
\begin{align}\label{f_e}
f_e \colon C_e \to \beta (e) \subset \mathbb{G}(r, n). 
\end{align}
Note that $f_e$ is a finite map 
between projective lines. 
We obtain the exact sequence, 
\begin{align}\label{taut}
0\to S \to \oO_{C_e}^{\oplus n} \stackrel{q}{\to} Q \to 0,
\end{align}
hence a quotient $q$, 
 by pulling back the tautological 
sequence on $\mathbb{G}(r, n)$ to $C_e$.
Let $v$ and $v'$ be the two vertices 
incident to $e$, 
and $x, x' \in C_e$ the corresponding 
ramification 
points respectively. 
We have the following cases. 

(i) Suppose that both of $v$ and $v'$ 
 satisfy (\ref{pos}). 
Then we take the quotient $q$. 

(ii) Suppose that exactly one of
$v$ or $v'$, say $v$,  
 does not satisfy (\ref{pos}). 
For simplicity, we assume that $\iota_{v}(j)=j$
for $1\le j \le n$, and
\begin{align*}
\iota_{v'}(j)=j, \quad 1\le j\le r-1, \quad 
\iota_{v'}(r)=r+1.
\end{align*} 
Then the exact sequence (\ref{taut}) is identified 
with the sequence,  
\begin{align}\label{iden:seq}
0 \to \oO_{C_e}^{\oplus r-1} \oplus \oO_{C_e}(-\delta(e)) \to 
\oO_{C_e}^{\oplus n} \to \oO_{C_e}(\delta(e)) 
\oplus \oO_{C_e}^{\oplus n-r-1} \to 0. 
\end{align}
Here the embedding 
\begin{align*}
\oO_{C_e}^{\oplus r-1} \oplus \oO_{C_e}(-\delta(e)) 
\subset \oO_{C_e}^{\oplus n},
\end{align*}
is the composition, 
\begin{align*}
\oO_{C_e}^{\oplus r-1} \oplus \oO_{C_e}(-\delta(e))
\subset \oO_{C_e}^{\oplus r-1}\oplus \oO_{C_e}^{\oplus 2}
\subset \oO_{C_e}^{\oplus n},
\end{align*}
where the first embedding is the direct sum of the identity and 
the pull-back of the tautological embedding via $f_e$, and 
the second one is the embedding into the first $r+1$-factors. 
Composing the embedding 
\begin{align*}
0 \to \bigoplus_{i=1}^{r-1} \oO_{C_e}(-s_i(v)x) \oplus 
\oO_{C_e}(-s_r(v)x-\delta(e)) \to 
\oO_{C_e}^{\oplus r-1} \oplus \oO_{C_e}(-\delta(e)),
\end{align*}
with the sequence (\ref{iden:seq}), we obtain the 
exact sequence, 
\begin{align*}
0 \to \bigoplus_{i=1}^{r-1} \oO_{C_e}(-s_i(v)x) \oplus 
\oO_{C_e}(-s_r(v)x-\delta(e)) \to \oO_{C_e}^{\oplus n} 
\stackrel{q'}{\to} Q' \to 0.
\end{align*}
Then we take the quotient $q'$. 

(iii) Suppose that both $v$ and $v'$ 
do not satisfy (\ref{pos}). 
Then as above, we take the exact sequence, 
\begin{align*}
0 \to \bigoplus_{i=1}^{r-1} \oO_{C_e}(-s_i(v)x-s_i(v')x') \oplus 
& \oO_{C_e}(-s_r(v)x-s_r(v')x'-\delta(e)) \\
& \to \oO_{C_e}^{\oplus n} 
\stackrel{q''}{\to} Q'' \to 0, 
\end{align*}
and we take the quotient $q''$. 
\end{itemize}
By gluing the above quotients, we obtain 
a curve $C$ and a quotient from $\oO_C^{\oplus n}$. 
The condition (\ref{graph:ample}) ensures 
that the resulting quotient is $\epsilon$-stable. 
\subsection{Virtual localization formula}\label{sub:vir}
Let $Q^{T}(\theta)$ be
the $T$-fixed locus (\ref{wei:poi}), and 
$i_{\theta}$ the inclusion, 
\begin{align*}
i_{\theta}\colon Q^{T}(\theta) \hookrightarrow 
\overline{Q}_{g, m}^{\epsilon}(\mathbb{G}(r, n), d). 
\end{align*}
We denote by $N^{\vir}(\theta)$ the virtual normal bundle of 
$Q^{T}(\theta)$ in $\overline{Q}_{g, m}^{\epsilon}(\mathbb{G}(r, n), d)$.
The virtual localization formula~\cite{GP}
in this case is written as 
\begin{align}\label{vloc}
[\overline{Q}_{g, m}^{\epsilon}(\mathbb{G}(r, n), d)]^{\vir}
&=\sum_{\theta}i_{\theta !}\left(\frac{[Q^{T}(\theta)]}{{\bf e}
(N^{\vir}(\theta))}\right) \\ \notag
& 
\in A_{\ast}^{T}(\overline{Q}_{g, m}^{\epsilon}
(\mathbb{G}(r, n), d), \mathbb{Q})\otimes_{R} \mathbb{Q}(\lambda_1, \cdots, \lambda_n). 
\end{align}
Here $R$ is the equivariant Chow ring of a point
with respect to the trivial $T$-action, 
\begin{align*}
R=\mathbb{Q}[\lambda_1, \cdots, \lambda_n], 
\end{align*}
with $\lambda_i$ equivariant parameters. 

Let $v$ be a vertex in the data (\ref{ind:graph})
which satisfies the condition (\ref{pos}). 
We see the contribution of $v$ to the RHS of (\ref{vloc}). 
For simplicity, we assume that $\iota_v(j)=j$ for $1\le j\le r$. 
The vertex $v$ corresponds to the space
\begin{align*}
\overline{M}_{\gamma(v), \valence(v)|{\bf s}(v)}^{\epsilon}/\Pi_{i=1}^{r}
S_{s_i(v)}.
\end{align*}
Similarly to the sequence (\ref{tuple}), each 
point on the above space 
corresponds to an exact sequence, 
\begin{align*}
0 \to S=\oplus_{i=1}^{r}S_i \to \oO_{C}^{\oplus n} \to Q \to 0, 
\end{align*}
for $S_i=\oO_{C_v}(-D_i)$ with $\deg(D_i)=s_i(v)$,
and $\valence(v)$-marked points. 
The exact sequence (\ref{tanob}) and the argument of~\cite[Section~7]{MOP}
show that the contribution of the vertex $v$ is 
\begin{align}\label{cont1}
\Cont(v)&= \frac{{\bf e}(\mathbb{E}^{\ast}\otimes T_{\nu(v)})}{{\bf e}(T_{\nu(v)})}
\frac{1}{\prod_{e}\frac{\lambda(e)}{\delta(e)}-\psi_e} \\
\label{cont2}
& \quad \frac{1}{\prod_{i\neq j}{\bf e}(H^0(O_C(S_i)|_{S_j})\otimes [\lambda_j-\lambda_i])} \cdot \\
\label{cont3}
& \quad \frac{1}{\prod_{i\neq j^{\ast}}{\bf e}(H^0(O_C(S_i)|_{S_i})\otimes 
[\lambda_{j^{\ast}}-\lambda_i])}.
\end{align}
Here each factor is as follows. 
\begin{itemize}
\item 
The symbol ${\bf e}$ denotes the Euler class, $T_{\nu(v)}$ is the $T$-representation on the tangent space of $\mathbb{G}(r, n)$ at $\nu(v)$, and $\mathbb{E}$ is the Hodge bundle, 
\begin{align}\label{Hodge}
\mathbb{E} \to \overline{M}_{\gamma(v), \valence(v)|{\bf s}(v)}^{\epsilon}. 
\end{align}
\item The product in the denominator of (\ref{cont1}) is over all half-edges $e$ incident to $v$. 
The factor $\lambda(e)$ denotes the $T$-weight of the tangent representation 
along the corresponding $T$-fixed edge, and $\psi_e$ 
is the first chern class of the cotangent line at the corresponding marking of
$\overline{M}_{\gamma(v), \valence(v)|{\bf s}(v)}^{\epsilon}$. 
(See (\ref{1stch}) below.)
\item The products in (\ref{cont2}), (\ref{cont3}) 
satisfy the following conditions, 
\begin{align*}
1\le i\le r, \quad 1 \le j \le r, \quad r+1 \le j^{\ast} \le n. 
\end{align*}
The brackets $[\lambda_j-\lambda_i]$ denotes the 
trivial bundle with specified weights. 
\end{itemize}

The same argument 
describing $\Cont(v)$ as above is also 
applied to see the contribution term of the edge
$e$ to the formula (\ref{vloc}). 
However we do not need to know 
its precise formula, and it is enough to notice that 
\begin{align*}
\Cont(e) \in \mathbb{Q}(\lambda_1, \cdots, \lambda_n). 
\end{align*}
The above fact is easily seen by the description of the 
$T$-fixed $\epsilon$-stable quotients in the last 
subsection. 
Then the RHS of (\ref{vloc}) is the sum of the products, 
\begin{align*}
\sum_{\theta}i_{\theta !}\left(
\prod_{e}\Cont(e)\prod_{v}\Cont(v)[Q^{T}(\theta)]\right). 
\end{align*}
\subsection{Classes on $\overline{M}_{g, m|d}^{\epsilon}$}
\label{subsec:Class}
As we have seen, each term of the virtual 
localization formula is a class on the moduli space of weighted 
pointed stable curves $\overline{M}_{g, m|d}^{\epsilon}$. 
The relevant classes on $\overline{M}_{g, m|d}^{\epsilon}$ for a 
sufficiently small $\epsilon$
is discussed in~\cite[Section~4]{MOP}. 
For arbitrary $0<\epsilon \le 1$
the similar classes are also available, which we 
recall here. 

For every subset $J\subset \{1, \cdots, d\}$ of size at least $2$, 
there is a diagonal class, 
\begin{align*}
D_{J} \in A^{|J|-1}(\overline{M}_{g, m|d}^{\epsilon}, \mathbb{Q}), 
\end{align*}
corresponding to the weighted pointed stable curves 
\begin{align*}
(C, p_1, \cdots, p_m, \widehat{p}_1, \cdots, \widehat{p}_d)
\end{align*}
satisfying 
\begin{align*}
\widehat{p}_j=\widehat{p}_{j'}, \quad j, j' \in J. 
\end{align*}
Note that $D_J=0$ if $\epsilon \cdot \lvert J \rvert >1$. 

Next we have the cotangent line bundles, 
\begin{align*}
\mathbb{L}_i \to \overline{M}_{g, m|d}^{\epsilon}, \quad
\widehat{\mathbb{L}}_j \to \overline{M}_{g, m|d}^{\epsilon},
\end{align*}
for $1\le i\le m$ and $1\le j \le d$, corresponding to 
the respective markings. 
We have the associated first chern classes, 
\begin{align}\label{1stch}
\psi_{i}=c_1(\mathbb{L}_i), \
\widehat{\psi}_{j}=c_1(\widehat{\mathbb{L}}_j) \in
 A^1(\overline{M}_{g, m|d}^{\epsilon}, \mathbb{Q}). 
\end{align}
The above classes are related as follows. 
For a subset $J\subset \{1, \cdots, d\}$, the class
\begin{align}\label{depJ}
\widehat{\psi}_{J}\cneq \widehat{\psi}_{j}|_{D_J},
\end{align}
does not depend on $j\in J$. If $J$ and $J'$ have non-trivial 
intersections, it is easy to see that 
\begin{align}\label{psiD}
D_{J}\cdot D_{J'}=(-\widehat{\psi}_{J\cup J'})^{\lvert J\cap J' \rvert -1}
D_{J\cup J'}. 
\end{align}
By the above properties, we obtain the notion of \textit{canonical forms},
(cf.~\cite[Section~4]{MOP},)
for any monomial $M(\widehat{\psi}_j, D_J)$ of 
$\widehat{\psi}_j$ and $D_J$. 
It is obtained as follows. 
\begin{itemize}
\item We multiply the classes $D_J$ using the formula (\ref{psiD}) until 
we obtain the product of classes $\widehat{\psi}_j$ and $D_{J_1}D_{J_2}\cdots D_{J_l}$ with all $J_i$ disjoint. 
\item Using (\ref{depJ}), we collect the equal cotangent classes. 
\end{itemize}
By extending the above operation linearly, we
obtain the canonical form for any polynomial 
$P(\widehat{\psi}_j, D_{J})$. 

\subsection{Standard classes under change of $\epsilon$}
For $\epsilon \ge \epsilon'$, recall that there is a 
birational morphism, (cf.~(\ref{nat:bir}), Remark~\ref{rmk:coin},) 
\begin{align*}
c_{\epsilon, \epsilon'} \colon \overline{M}_{g, m|d}^{\epsilon}
\to \overline{M}_{g, m|d}^{\epsilon'}. 
\end{align*}
For simplicity, we write $c_{\epsilon, \epsilon'}$ as $c$. 
Then we have the following, 
\begin{align}\label{form1}
c^{\ast}\psi_i &=\psi_i, \quad 1\le i\le m, \\
\label{form2}
c^{\ast}\widehat{\psi}_j &=\widehat{\psi}_j -\Delta_j, \quad 1\le j\le d. 
\end{align}
Here $\Delta_{j}$ is given by 
\begin{align}\label{sum:delta}
\Delta_{j}=\sum_{j\in J\subset \{1, \cdots, d\}}\Delta_{J}, 
\end{align}
where $\Delta_{J} \subset \overline{M}_{g, m|d}^{\epsilon}$ 
correspond to curves 
\begin{align*}
C=C_1\cup C_2, \quad g(C_1)=0, \quad g(C_2)=g, 
\end{align*}
with a single node which separates $C_1$ and $C_2$, and 
the markings of $J$ are distributed to $C_1$. 
The subsets $J$ in the sum (\ref{sum:delta}) should satisfy, 
\begin{align*}
&\epsilon \cdot \lvert J \rvert -1>0, \\
&\epsilon' \cdot \lvert J \rvert -1 \le 0. 
\end{align*}
Applying (\ref{form1}), (\ref{form2}) and the 
projection formula, we obtain the universal formula,
\begin{align}\label{univ:form}
c_{\ast}\left(\prod_{i=1}^{m}\psi_i^{m_i}\prod_{j=1}^{d}\widehat{\psi}_j^{n_j}
 \right)=\prod_{i=1}^{m}\psi_i^{m_i}
 \left(\prod_{j=1}^{d}\widehat{\psi}_j^{n_j}+
\cdots \right). 
\end{align}
If $\epsilon=1$ and $0<\epsilon' \ll 1$, the above 
formula coincides with the formula obtained in~\cite[Lemma~3]{MOP}. 

Also the Hodge bundle (\ref{Hodge}) satisfies 
\begin{align}\label{sat:Hodge}
c^{\ast}\mathbb{E}\cong \mathbb{E},
\end{align}
since $c$ contracts only rational tails. 
\subsection{The case of genus zero}
In genus zero, note that the moduli space 
\begin{align*}
\overline{Q}_{0, m}^{\epsilon}(\mathbb{G}(r, n), d)
\end{align*}
is non-singular by Lemma~\ref{lem:nonsing}. 
 If it is also connected, then 
it is irreducible and there is a birational map, 
\begin{align*}
\overline{Q}_{0, m}^{\epsilon_1}(\mathbb{G}(r, n), d)
\dashrightarrow \overline{Q}_{0, m}^{\epsilon_2}(\mathbb{G}(r, n), d),
\end{align*}
as long as $\epsilon_i>(2-m)/d$. 
In fact, we have the following. 
\begin{lem}\label{lem:eqconn}
The moduli stack 
\begin{align}\label{eq:conn}
\overline{Q}_{g, m}^{\epsilon}(\mathbb{G}(r, n), d)
\end{align}
is connected. 
\end{lem}
\begin{proof}
The connectedness of the stable map moduli spaces
is proved in~\cite{KimPan}, and we reduce the connectedness of (\ref{eq:conn})
to that of the stable map moduli spaces. 
To do this, it is enough to see that any $\epsilon$-stable 
quotient $q\colon \oO_C^{\oplus n} \twoheadrightarrow Q$
is deformed to a quotient obtained by a stable map. 
By applying the $T$-action, we may assume that $q$ is 
a $T$-fixed quotient. Then $q$ fits into 
an exact sequence, 
\begin{align*}
0 \to \oplus_{i=1}^{r}\oO_C(-D_i) \to \oO_{C}^{\oplus n} \stackrel{q}{\to}
Q \to 0, 
\end{align*}
for $r$-tuple divisors $D_i$ on $C$. (See Subsection~\ref{subsec:locl}.)
By deforming $D_i$ to reduced divisors $D_i'$, we can 
deform the quotient $q$ to $q' \colon \oO_{C}\twoheadrightarrow Q'$
which is $\epsilon$-stable
for $\epsilon=1$.  
Then by Lemma~\ref{lem:worth}, we can deform $q'$ to a quotient 
corresponding to a stable map. 
\end{proof}
The smoothness of the genus zero moduli spaces
 and the above lemma show the 
formula,
\begin{align}\label{wcf:0}
c_{\epsilon, \epsilon' \ast}\iota_{\ast}^{\epsilon}
\left([\overline{Q}_{0, m}^{\epsilon}(\mathbb{G}(r, n), d)]^{\vir} \right)=
\iota_{\ast}^{\epsilon'}\left(
[\overline{Q}_{0, m}^{\epsilon'}(\mathbb{G}(r, n), d)]^{\vir}\right).
\end{align}
Hence Theorem~\ref{thm:wcf} in the genus zero case is proved. 
\subsection{Sketch of the proof of Theorem~\ref{thm:wcf}}
Under the map to 
$\overline{Q}_{g, m}^{\epsilon'}(\mathbb{G}(1, \dbinom{n}{r}), d)$, 
several rational tails on 
$\overline{Q}_{g, m}^{\epsilon}(\mathbb{G}(r, n), d)$
with small degree collapse. Also the $T$-fixed loci of 
$\overline{Q}_{g, m}^{\epsilon}(\mathbb{G}(r, n), d)$
have many splitting types of the subbundle $S$
which are collapsed. 
For a non-collapsed edge, its contribution exactly 
coincides, and we just need to show the matching on
each vertex. 

The equality (\ref{wcf:0})
implies that both sides are equal after 
$T$-equivariant localization. 
For each vertex $v$
on $\overline{Q}_{0, m}^{\epsilon}(\mathbb{G}(r, n), d)$, 
the contribution
\begin{align*}
\Cont(v) \in A_{\ast}^{T}(\overline{M}^{\epsilon}_{0, \valence(v)|
{\bf s}(v)}, \mathbb{Q})\otimes_{R}\mathbb{Q}(\lambda_1, \cdots, \lambda_n),
\end{align*}
is given in Subsection~\ref{sub:vir}.
In genus zero the Hodge bundle is trivial, and  
the class $\Cont(v)$ is easily seen to be written 
as an element, 
\begin{align*}
\Cont(v) \in \mathbb{Q}(\lambda_1, \cdots, \lambda_n)[\psi_i, \widehat{\psi}_j, D_J],
\end{align*} 
 symmetric with 
respect to the variables $\widehat{\psi}_{j}$. 
Let us take the push forward to 
$\overline{Q}_{0, m}^{\epsilon'}(\mathbb{G}(1, \dbinom{n}{r}), d)$
using (\ref{univ:form}),
and take the canonical form. (cf.~Subsection~\ref{subsec:Class}.)
At each vertex on 
$\overline{Q}_{0, m}^{\epsilon'}(\mathbb{G}(1, \dbinom{n}{r}), d)$,
the vertices and the collapsed edges
on $\overline{Q}_{0, m}^{\epsilon}(\mathbb{G}(r, n), d)$ 
contribute to the LHS of (\ref{wcf:0}) by  
the polynomial, 
\begin{align*}
L^{C}(\psi_i, \widehat{\psi}_j, D_J). 
\end{align*} 
Also 
the vertices on $\overline{Q}_{0, m}^{\epsilon'}(\mathbb{G}(r, n), d)$
with collapsed splitting types 
contribute to the RHS of (\ref{wcf:0}) by the polynomial, 
\begin{align*}
R^{C}(\psi_i, \widehat{\psi}_j, D_J).
\end{align*}
The equality (\ref{wcf:0})
implies the equality, 
\begin{align}\label{match}
L^{C}(\psi_i, \widehat{\psi}_j, D_J)=
R^{C}(\psi_i, \widehat{\psi}_j, D_J),
\end{align}
as \textit{classes}
in the equivariant Chow ring. 

Although (\ref{match}) is an equality 
after taking classes, the 
exactly same argument of~\cite[Lemma~5]{MOP}
shows that the equality (\ref{match}) 
holds as \textit{abstract polynomials}.
Also note that the
genus dependent part involving Hodge bundles (\ref{cont1})
in the virtual 
localization formula (\ref{vloc}) does not depend on $\epsilon$
by (\ref{sat:Hodge}).
Therefore the above argument   
immediately implies 
\begin{align}\label{wcf:g}
c_{\epsilon, \epsilon' \ast}\iota_{\ast}^{\epsilon}
\left([\overline{Q}_{g, m}^{\epsilon}(\mathbb{G}(r, n), d)]^{\vir} \right)=
\iota_{\ast}^{\epsilon'}\left(
[\overline{Q}_{g, m}^{\epsilon'}(\mathbb{G}(r, n), d)]^{\vir}\right),
\end{align}
for any $g\ge 0$. Hence we obtain the formula (\ref{WCF}). 

\section{Invariants on (local) Calabi-Yau 3-folds}\label{sec:enu}
In this section, we introduce some enumerative invariants 
of curves 
on (local) Calabi-Yau 3-folds and propose 
related problems. 
Similar invariants for MOP-stable 
quotients are discussed in~\cite[Section~9, 10]{MOP}.
In what follows, we use the notation 
(\ref{univ1}), (\ref{univ2}) for 
universal curves and quotients. 
\subsection{Invariants on a local $(-1, -1)$-curve}
Let us consider a crepant small resolution of a 
conifold singularity, that is the total space of 
$\oO_{\mathbb{P}^{1}}(-1)^{\oplus 2}$, 
\begin{align*}
X=\oO_{\mathbb{P}^1}(-1) \oplus \oO_{\mathbb{P}^1}(-1) \to \mathbb{P}^1.
\end{align*}
In a similar way to~\cite[Section~9]{MOP}, 
we define the $\mathbb{Q}$-valued invariant by
\begin{align}\label{inv1}
N_{g, d}^{\epsilon}(X)\cneq 
\int_{[\overline{Q}_{g, 0}^{\epsilon}(\mathbb{P}^1, d)]^{\vir}}
{\bf e}(R^1 \pi^{\epsilon}_{\ast}(S_{U^{\epsilon}})\oplus 
R^1 \pi^{\epsilon}_{\ast}(S_{U^{\epsilon}})).
\end{align}
It is easy to see that
\begin{align*}\pi^{\epsilon}_{\ast}(S_{U^{\epsilon}})=0,
\end{align*}
 hence $R^1 \pi^{\epsilon}_{\ast}(S_{U^{\epsilon}})$
 is a vector bundle and (\ref{inv1}) is well-defined. 
By Theorem~\ref{thm:cross} (i) and Lemma~\ref{lem:empty}, we have 
\begin{align*}
N_{g, d}^{\epsilon}(X)&=N_{g, d}^{\rm{GW}}(X), \quad \epsilon>2, \\
N_{0, d}^{\epsilon}(X)&=0, \quad 0<\epsilon \le 2/d. 
\end{align*}
Here $N_{g, d}^{\rm{GW}}(X)$ is the 
genus $g$, degree $d$ 
local GW invariant 
of $X$. The following result is 
obtained by the same method of~\cite[Proposition~6, 7]{MOP}, 
using the localization with respect to the 
twisted $\mathbb{C}^{\ast}$-action on $X$, and the vanishing result 
similar to~\cite{FaPan}.
We leave the readers to check the detail. 
\begin{thm}\label{thm:check}
We have the following.
\begin{align*}
N_{g, d}^{\epsilon}(X)=\left\{ \begin{array}{cc}
N_{g, d}^{\rm{GW}}(X), & 2g-2+\epsilon \cdot d>0, \\
0, & 2g-2+\epsilon \cdot d \le 0. \end{array}\right. 
\end{align*}
\end{thm}
Let $F^{\rm{GW}}(X)$ be the generating series, 
\begin{align*}
F^{\rm{GW}}(X)=\sum_{g\ge 0, \ d>0}N_{g, d}^{\rm{GW}}(X)\lambda^{2g-2}t^d.
\end{align*}
Recall that we have the 
 following Gopakumar-Vafa formula, 
\begin{align}\label{GV}
F^{\rm{GW}}(X)=\sum_{d\ge 1}\frac{t^d}{4d \sin^2 (d\lambda/2)}. 
\end{align}
By Theorem~\ref{thm:check} and the formula (\ref{GV}),
the generating series of $N_{g, d}^{\epsilon}(X)$ satisfies 
the formula, 
\begin{align*}
F^{\epsilon}(X)&\cneq 
\sum_{g\ge 0, \ d>0}N_{g, d}^{\epsilon}(X)\lambda^{2g-2}t^d \\
&=\sum_{d\ge 1}\frac{t^d}{4d \sin^2 (d\lambda/2)}
-\sum_{0<d\le 2/\epsilon}\frac{1}{d^3}\lambda^{-2}t^d. 
\end{align*} 
\subsection{Invariants on a local projective plane}
Let us consider the total space of the canonical line 
bundle of $\mathbb{P}^2$, 
\begin{align*}
X=\oO_{\mathbb{P}^2}(-3) \to \mathbb{P}^2. 
\end{align*}
As in the case of a $(-1, -1)$-curve, we can define the invariant by 
\begin{align*}
N_{g, d}^{\epsilon}(X)\cneq 
\int_{[\overline{Q}_{g, 0}^{\epsilon}(\mathbb{P}^2, d)]^{\vir}}
{\bf e}(R^1 \pi^{\epsilon}_{\ast}(S_{U^{\epsilon}}^{\otimes 3}))\in \mathbb{Q},
\end{align*}
since we have the vanishing, 
\begin{align*}
\pi^{\epsilon}_{\ast}(S_{U^{\epsilon}}^{\otimes 3})=0. 
\end{align*}
Note that $N_{g, d}^{\epsilon}(X)$ is a local 
GW invariant of $X$ when $\epsilon>2$. 
However for a small $\epsilon$, 
the following example shows that 
$N_{g, d}^{\epsilon}(X)$ is different from 
the local GW invariant of $X$.  
\begin{exam}
For $X=\oO_{\mathbb{P}^2}(-3)$, 
an explicit computation shows that 
\begin{align*}
N_{1, 1}^{\epsilon}(X)=\left\{
\begin{array}{cc}
\frac{1}{4}, & \epsilon>1, \\
\frac{3}{4}, & 0<\epsilon \le 1. 
\end{array}
 \right. 
\end{align*}
In fact if $\epsilon>1$, then 
$N_{1, 1}^{\epsilon}(X)$ coincides with the local 
GW invariant of $X$, and it is 
already computed.
A list is available in~\cite[Table~1]{AMV}
in a Gopakumar-Vafa form. 

Let us compute $N_{1, 1}^{\epsilon}(X)$ for 
 $0<\epsilon \le 1$. In this case, any $\epsilon$-stable 
quotient of type $(1, 3, 1)$ is MOP-stable, and 
the moduli space is described as
\begin{align*}
\overline{Q}_{1, 0}^{\epsilon}(\mathbb{P}^2, 1)\cong 
\overline{M}_{1, 1}\times \mathbb{P}^2. 
\end{align*}
(cf.~\cite[Example~5.4]{MOP}.)
Also there is no obstruction in this case, 
\begin{align*}
[\overline{Q}_{1, 0}^{\epsilon}(\mathbb{P}^2, 1)]^{\vir}
=[\overline{Q}_{1, 0}^{\epsilon}(\mathbb{P}^2, 1)].
\end{align*}
Let 
\begin{align*}
\pi \colon U \to \overline{M}_{1, 1},
\end{align*}
be the universal curve with a section 
$D\subset U$.
Then 
\begin{align*}
U^{\epsilon}=U\times \mathbb{P}^2 \to 
\overline{Q}_{1, 0}^{\epsilon}(\mathbb{P}^2, 1),
\end{align*}
is the universal curve, and the universal subsheaf 
$S_{U^{\epsilon}} \subset \oO_{U^{\epsilon}}^{\oplus 3}$
is given by 
\begin{align*}
S_{U^{\epsilon}}\cong \oO_{U}(-D)\boxtimes \oO_{\mathbb{P}^2}(-1). 
\end{align*}
Therefore we have 
\begin{align*}
R^1 \pi_{\ast}^{\epsilon}(S_{U^{\epsilon}}^{\otimes 3})
\cong R^1 \pi_{\ast}\oO_{U}(-3D) \boxtimes \oO_{\mathbb{P}^2}(-3). 
\end{align*}
The vector bundle $R^1\pi_{\ast}\oO_U(-3D)$ 
on $\overline{M}_{1, 1}$ 
admits a filtration whose subquotients are
line bundles $\mathbb{E}^{\vee}$, 
$\mathbb{L}_1$ and $\mathbb{L}_1^{\otimes 2}$. 
Therefore the integration of the Euler class is given by 
\begin{align*}
\int_{\overline{Q}_{1, 0}^{\epsilon}(\mathbb{P}^2, 1)}{\bf e}(R^1 \pi_{\ast}^{\epsilon}(S_{U^{\epsilon}}^{\otimes 3}))
&=9\cdot \int_{\overline{M}_{1, 1}}(3\psi_1 -c_1(\mathbb{E})), \\
&=\frac{3}{4}.
\end{align*}
Here the last equality follows from the computation in~\cite{FaPan}, 
\begin{align*}
\int_{\overline{M}_{1, 1}}c_1(\mathbb{E})=
\int_{\overline{M}_{1, 1}}\psi_1=\frac{1}{24}. 
\end{align*}
\end{exam}
By the above example, 
the following problem seems to be interesting. 
\begin{prob}\label{prob1}
How do the invariants $N_{g, d}^{\epsilon}(X)$ depend on $\epsilon$,
when $X=\oO_{\mathbb{P}^2}(-3)$?
\end{prob}

\subsection{Generalized tree level GW 
systems on hypersurfaces}
Let $X$ be a smooth projective variety, defined 
by the degree $N$ homogeneous polynomial $f$
of $n+1$ variables, 
\begin{align*}
X=\{f=0\} \subset \mathbb{P}^{n}. 
\end{align*}
Recall that in Lemma~\ref{lem:nonsing}, 
the moduli stack 
$\overline{Q}_{0, m}^{\epsilon}(\mathbb{P}^n, d)$
is shown to be smooth of the expected dimension. 
We construct the closed substack, 
\begin{align}\label{close}
\overline{Q}_{0, m}^{\epsilon}(X, d)\subset
\overline{Q}_{0, m}^{\epsilon}(\mathbb{P}^n, d),
\end{align}
as follows. For an $\epsilon$-stable quotient 
of type $(1, n+1, d)$, 
\begin{align*}
0 \to S \to \oO_{C}^{\oplus n+1} \to Q \to 0, 
\end{align*}
we take the dual of the first inclusion, 
\begin{align*}
(s_0, s_1, \cdots, s_n) \colon \oO_{C}^{\oplus n+1} 
\to S^{\vee}.
\end{align*}
Applying $f$, we obtain the section, 
\begin{align}\label{fs}
f(s_0, s_1, \cdots, s_n) \in H^0(C, S^{\otimes -N}). 
\end{align}
In genus zero, we have the vanishing
\begin{align*}
R^1 \pi_{\epsilon}^{\ast}(S^{\otimes -N}_{U^{\epsilon}})=0, 
\end{align*}
hence 
(\ref{fs}) determines a section of 
the vector bundle
$\pi_{\ast}^{\epsilon}(S^{\otimes -N}_{U^{\epsilon}})$, 
which we denote 
\begin{align*}
s_f \in H^0(\overline{Q}_{0, m}(\mathbb{P}^n, d), 
\pi_{\ast}^{\epsilon}(S^{\otimes -N}_{U^{\epsilon}})). 
\end{align*}
Then we define (scheme theoretically)
\begin{align}\label{QX}
\overline{Q}_{0, m}^{\epsilon}(X, d)=
\{s_{f}=0\}. 
\end{align}
Note that if $\epsilon>2$, then the above 
space coincides with the moduli stack of genus zero, 
degree $d$ stable 
maps to $X$. 
Since (\ref{QX}) is a zero locus of a section of a
vector bundle on a smooth stack, 
there is a perfect obstruction theory on it, 
determined by the two term complex, 
\begin{align*}
(\pi_{\ast}^{\epsilon}(S^{\otimes -N}))^{\vee}
\to \Omega_{\overline{Q}_{0, m}^{\epsilon}(\mathbb{P}^n, d)
}|_{\overline{Q}_{0, m}^{\epsilon}(X, d)}.
\end{align*}
The associated virtual class is denoted by 
\begin{align*}
[\overline{Q}_{0, m}^{\epsilon}(X, d)]^{\vir}
\in A_{\ast}(\overline{Q}_{0, m}^{\epsilon}(X, d), 
\mathbb{Q}).
\end{align*}
The evaluation map factors through $X$, 
\begin{align*}
\ev_i \colon \overline{Q}_{0, m}^{\epsilon}(X, d) \to X,
\end{align*}
for $1\le i \le m$. 
Hence we obtain the diagram, 
\begin{align*}
\xymatrix{
\overline{Q}_{0, m}^{\epsilon}(X, d) \ar[r]^{\alpha} 
\ar[d]_{(\ev_1, \cdots, \ev_m)}&
\overline{M}_{0, m} \\
X\times \cdots \times X, & }
\end{align*}
and a system of maps, 
\begin{align}\label{system}
I_{0, m, d}^{\epsilon} =&\alpha_{\ast}(\ev_1, \cdots, \ev_m)^{\ast}:
H^{\ast}(X, \mathbb{Q})^{\otimes m} \to H^{\ast}(\overline{M}_{0, m}, 
\mathbb{Q}). 
\end{align}
It is straightforward to 
check that the above system of maps (\ref{system}) 
satisfies the axiom of 
tree level GW system~\cite{KonMa}. 
In particular, 
we have the genus zero GW type invariants, 
\begin{align*}
\langle I_{0, m, d}^{\epsilon}\rangle
(\gamma_1 \otimes \cdots \otimes \gamma_m)
=\int_{\overline{M}_{0, m}}I_{0, m, d}^{\epsilon}
(\gamma_1 \otimes \cdots \otimes \gamma_m), 
\end{align*}
for $\gamma_i \in H^{\ast}(X, \mathbb{Q})$.
The formal function 
\begin{align*}
\Phi^{\epsilon}(\gamma)
=\sum_{m\ge 3, \ d\ge 0}
\frac{1}{n!}
\langle I_{0, m, d}^{\epsilon} \rangle
(\gamma^{\otimes m})q^d,
\end{align*}
satisfies the WDVV equation~\cite{KonMa}, 
and induces the generalized 
big (small) quantum cohomology ring, 
\begin{align*}
(H^{\ast}(X, \mathbb{Q})\db[ q \db], \circ^{\epsilon}),
\end{align*}
depending on $\epsilon\in \mathbb{R}_{>0}$. 
For $\epsilon>2$, the above ring 
coincides with the big (small) quantum cohomology 
ring defined by the GW theory on $X$. 
\begin{rmk}
The above construction of the generalized 
tree level GW system 
can be easily generalized to  
any complete intersection of the Grassmannian
$X\subset \mathbb{G}(r, n)$.  
\end{rmk}
\begin{rmk}
As discussed 
in~\cite[Section~10]{MOP} for MOP-stable quotients, 
it might be possible to
define the substack (\ref{close}) 
and the virtual class on it for every genera. 
\end{rmk}
\subsection{Enumerative invariants on 
projective Calabi-Yau 3-folds}
The construction in the previous subsection 
enables us to construct 
genus zero GW type invariants 
without point insertions
on several projective Calabi-Yau 3-folds. 
One of the interesting examples is 
a quintic 3-fold, 
\begin{align*}
X \subset \mathbb{P}^4. 
\end{align*}
We can define the invariant, 
\begin{align}\notag
N_{0, d}^{\epsilon}(X)&=
\int_{[\overline{Q}_{0, d}^{\epsilon}(X, d)]^{\vir}}1 \\
\label{quin:inv}
&=\int_{\overline{Q}_{0, d}^{\epsilon}(\mathbb{P}^4, d)}
{\bf e} \left(\pi_{\ast}^{\epsilon}(S_{U^{\epsilon}}^{\vee \otimes 5})\right)
\in \mathbb{Q}. 
\end{align}
Another interesting example is a Calabi-Yau 3-fold 
obtained as a complete intersection of the 
Grassmannian $\mathbb{G}(2, 7)$. 
Let us consider the Pl$\ddot{\rm{u}}$cker embedding, 
\begin{align*}
\mathbb{G}(2, 7)\hookrightarrow \mathbb{P}^{20}, 
\end{align*}
and take general hyperplanes 
\begin{align}\label{hype}
H_1, \cdots, H_7 \subset \mathbb{P}^{20}.
\end{align}
Then the intersection 
\begin{align*}
X =\mathbb{G}(2, 7)\cap H_1 \cap \cdots \cap H_7,
\end{align*}
is a projective Calabi-Yau 3-fold. 
The hyperplanes (\ref{hype}) 
defines the section, 
\begin{align*}
s_{H} \in H^0(\overline{Q}_{0, m}^{\epsilon}(\mathbb{G}(2, 7), d), 
\pi_{\ast}^{\epsilon}(\wedge^2 S^{\vee}_{U^{\epsilon}})^{\oplus 7}), 
\end{align*}
and we define 
\begin{align}\label{QXp}
\overline{Q}_{0, m}^{\epsilon}(X, d)
=\{s_H=0\}. 
\end{align}
As in the previous subsection, there is a 
perfect obstruction theory 
and the virtual class on (\ref{QXp}). 
In particular, we can define
\begin{align}\notag
N_{0, d}^{\epsilon}(X)&=
\int_{[\overline{Q}_{0, d}^{\epsilon}(X, d)]^{\vir}}1 \\
\label{grass:inv}
&=\int_{\overline{Q}_{0, d}^{\epsilon}(\mathbb{G}(2, 7), d)}
{\bf e} \left(\pi_{\ast}^{\epsilon}
(\wedge^{2}S_{U^{\epsilon}}^{\vee})^{\oplus 7}\right)
\in \mathbb{Q}. 
\end{align}
For $\epsilon>2$, both
 invariants (\ref{quin:inv}), (\ref{grass:inv})
coincide with the GW invariants
of $X$. 
As in Problem~\ref{prob1}, we can 
address the following problem. 
\begin{prob}
How do the invariants $N_{0, d}^{\epsilon}(X)$ depend on $\epsilon$,
when $X$ is a quintic 3-fold in $\mathbb{P}^4$
or a complete intersection of $\mathbb{G}(2, 7)$
of codimension $7$?
\end{prob}

Institute for the Physics and 
Mathematics of the Universe, University of Tokyo

\textit{E-mail address}:toda-914@pj9.so-net.ne.jp, 
yukinobu.toda@ipmu.jp

2000 Mathematics Subject Classification. 
14N35 (Primary); 14H60 (Secondary)

\end{document}